\documentclass[12pt]{amsart}

\usepackage{amssymb,stmaryrd}
\usepackage{rotating}

\usepackage{mathtools}

\usepackage{enumerate,pdflscape}
\usepackage{color}

   \newcommand{\sll}{  A_2   }

 \newcommand{\la}{\mathfrak{g}}

  \theoremstyle{definition}
  \newtheorem{definition}{Definition}[section]
  
   \newtheorem{remark}[definition]{Remark}
    \AtEndEnvironment{remark}{\null\hfill\qedsymbol}
  \theoremstyle{plain}

  \theoremstyle{plain}
  \newtheorem{lemma}[definition]{Lemma}
  
  \newtheorem{theorem}[definition]{Theorem}

\title[The subalgebras of $A_2$]{The  subalgebras of $A_2$}

\begin{document}

\author[Andrew Douglas]{Andrew Douglas$^{1,2}$}
\address[]{$^1$Department of Mathematics, New York City College of Technology, City University of New York, Brooklyn, NY, USA.}
\address[]{$^2$Ph.D. Programs in Mathematics and Physics, CUNY Graduate Center, City University of New York, New York, NY, USA.}

\author[Joe Repka]{Joe Repka$^3$}
\address{$^3$Department of Mathematics, University of Toronto, Toronto, ON, M5S 2E4, Canada}
\email{adouglas@citytech.cuny.edu, repka@math.toronto.edu}

\date{\today}

\keywords{Levi decomposable subalgebras, solvable subalgebras, classification of subalgebras, special linear algebra} 
\subjclass[2010]{17B05, 17B10, 17B20, 17B30}

\begin{abstract}
A classification of the semisimple subalgebras of the Lie algebra of traceless $3\times 3$ matrices with complex entries, denoted $A_2$,  is well-known.   We classify  its nonsemisimple subalgebras, thus completing the classification of the subalgebras of $\sll$.   
\end{abstract}

\maketitle

\section{Introduction}

The study of subalgebras of semisimple Lie algebras has largely focused on semsimple subalgebras.  Notable examples include the work of Dynkin \cite{dynkin}, and Minchenko \cite{min}, which combine to classify the simple subalgebras of the exceptional Lie algebras, up to inner automorphism.  Another important example is de Graaf's classification of the semisimple subalgebras of the simple Lie algebras of ranks $\le 8$, up to linear equivalence, which is somewhat weaker than a classification up to inner automorphism \cite{degraafc, degraafd}.

Less is known about nonsemisimple subalgebras of semisimple Lie algebras.  Levi's Theorem [\cite{levi}, Chapter III, Section 9] implies that subalgebras must be either semisimple, solvable, or 
Levi decomposable.  A subalgebra is  Levi decomposable if it is a semidirect or direct 
sum of a semisimple subalgebra and a solvable subalgebra.

We have made considerable progress towards classifying solvable and Levi decomposable subalgebras of semisimple Lie algebras.  For instance, in \cite{dkr}, we classified the abelian extensions of the special orthogonal Lie algebras $\mathfrak{so}(2n,\mathbb{C})$ in the exceptional Lie algebras $E_{n+1}$, up to inner automorphism. In \cite{abcd} we classified   an important family of Levi decomposable subalgebras of the classical Lie algebras. Most recently, we classified  all subalgebras of $A_1 \oplus A_1$, up to inner automorphism  \cite{drsof}.  $A_1$ is the special Lie algebra of traceless $2\times 2$ matrices with complex entries.

In the present article, we seek to extend the knowledge of subalgebras, especially nonsemisimple subalgebras, of semisimple Lie algebras.  In particular, we 
classify the solvable, and Levi decomposable subalgebras of the simple Lie algebra $A_2$ of traceless $3\times 3$ matrices with complex entries. Since its semisimple subalgebras are well-known \cite{degraafc, degraafd}, Levi's Theorem implies that the classification of the subalgebras of $A_2$  will be complete. The classification is up to inner automorphism.

The article is organized as follows.  Section \ref{sl} contains basic review of the simple Lie algebra $\sll$ and its semisimple subalgebras.  We describe a classification of solvable Lie algebras for degrees $\leq 3$ in Section \ref{solvable} that will be used in a proceeding section.
Additional notation and terminology that will be used in the article are recorded in Section \ref{extras}.   In Section \ref{levi}  we classify the Levi decomposable subalgebras of $\sll$.  Finally, in Section \ref{solvvv}, we classify the solvable subalgebras of $\sll$. 

\section{The simple Lie algebra $\sll$}\label{sl}

The special linear algebra $\sll$ is the Lie algebra of traceless $3 \times 3$ matrices with complex entries.     A  Chevalley basis $\{ x_i, y_i ,  h_j:  1\leq i \leq 3, 1 \leq j \leq 2 \}$ of $\sll$ is defined as follows:
\begin{equation}
\begin{array}{l}
\displaystyle  ah_1+ bh_2+  cx_1+ dx_2 +ex_3+  c'y_1+ d'y_2+ e'y_3\\=  \left(
\begin{array}{rrr}
 a & c & -e   \\
 c' & b-a & d   \\
 -e'& d' &-b  
\end{array}
\right).
\end{array} 
\end{equation}
where $a, b, c, d, e, c',d',e' \in \mathbb{C}$.

From \cite{degraafc,degraafd}, there are precisely two semisimple subalgebras of $\sll$ up to linear equivalence, both isomorphic to the special linear algebra $A_1$ of traceless $2\times 2$ matrices with complex entries.  By [\cite{min}, Theorem $3.3$], this is also true up to inner automorphism.  They are
\begin{equation}\label{semiss}
 \arraycolsep=1.4pt\def\arraystretch{1.3}
\begin{array}{llllllll}
A_{1}^1&\equiv&
\langle x_3, y_3, h_1+h_2   \rangle,\\
A_{1}^2&\equiv&
\langle x_1+x_2, 2y_1+2y_2, 2h_1+2h_2   \rangle.
\end{array}
\end{equation}

\section{Solvable Lie algebras of small dimension}\label{solvable}

A full classification of solvable Lie algebras is not known and thought to be an impossible task.  However, classifications of solvable Lie algebras in special cases have been considered
(e.g., \cite{degraafa, mubar, wintera, patera, snobl, wintb}).  For instance, de Graaf classified the solvable Lie algebras in dimension $\leq 4$ over fields of any characteristic \cite{degraafa}.  

With $\mathbb{C}$ as the ground field, we describe this classification in totality for dimension $\leq 3$,  and  describe that portion of the dimension $4$ classification which is relevant to the present article.  In each case, $z_1, z_2,...,z_n$ are basis elements of the $n$-dimensional, solvable Lie algebra being described, and all nonzero commutation relations for that algebra are presented.
\begin{equation}\label{degall}
 \arraycolsep=1.4pt\def\arraystretch{1.3}
\begin{array}{llll}
J & \text{The abelian Lie algebra of dimension $1$}\\
\end{array}
\end{equation}
\begin{equation} \arraycolsep=1.4pt\def\arraystretch{1.3}
\begin{array}{llll}
K_1 & \text{The abelian Lie algebra of dimension $2$}\\
K_2 & [z_1,z_2]=z_1 \\
\end{array}
\end{equation}
\begin{equation}\label{Eq:deg3}
 \arraycolsep=1.4pt\def\arraystretch{1.3}
\begin{array}{llll}
L_1 & \text{The abelian Lie algebra of dimension $3$}\\
L_2 & [z_3,z_1]=z_1, [z_3,z_2]=z_2 \\
L_{3,a} & [z_3,z_1]=z_2, [z_3,z_2]=az_1+z_2 \\
L_{4} & [z_3,z_1]=z_2, [z_3,z_2]=z_1 \\
L_{5} & [z_3,z_1]=z_2
\end{array}
\end{equation}
\begin{equation}\label{Eq:deg32}
 \arraycolsep=1.4pt\def\arraystretch{1.3}
\begin{array}{llll}
M_8 & [z_1, z_2]=z_2, [z_3, z_4]=z_4\\
M_{12} & [z_4, z_1]=z_1, [z_4, z_2]=2z_2, [z_4, z_3]=z_3, [z_3, z_1]=z_2\\
M_{13, a} & [z_4, z_1]=z_1+az_3, [z_4, z_2]=z_2, [z_4, z_3]=z_1, [z_3, z_1]=z_2\\
M_{14} & [z_4, z_1]=z_3, [z_4, z_3]=z_1, [z_3, z_1]=z_2\\
\end{array}
\end{equation}
Note that we get a nonisomorphic solvable Lie algebra $L_{3,a}$ for each $a\in \mathbb{C}$; and a
 nonisomorphic solvable Lie algebra $M_{13,a}$ for each $a\in \mathbb{C}$.

\section{Additional definitions and notation}\label{extras}

\begin{itemize}
\item Let $\la$ be a Lie algebra, then $Aut(\la)$ is the group of automorphisms of $\la$, $Inn(\la)$ is the group of inner automorphisms
of $\la$, and $Der(\la)$ is the Lie algebra of derivations of $\la$.
\item Let $\la$ and $\la'$ be Lie algebras, and $\varphi: \la \rightarrow Der(\la')$ a Lie algebra homomorphism.  Consider the Levi decomposable algebra $\la \inplus_{\varphi} \la'$ with $x\in \la$ and $y\in \la'$.  Then, $[x,y]\equiv \varphi(x)(y)$.
\item Let  $W_1,..., W_m \in \sll$. Then, 
\[
\langle W_1,..., W_m \rangle_{}
\]
is the subalgebra of $\sll$ generated by $W_1$,...,$W_m$.   In this article, $W_1$,...,$W_m$ will generally not be a minimal generating set.

\item Let $\varphi$ and $\varrho$ be Lie algebra embeddings of $\la'$ into $\mathfrak{g}$.  If $\varphi$ and $\varrho$ are  equivalent up to inner automorphism we write
$\varphi \sim \varrho$.
\item Two embeddings $\varphi$ and $\varrho$  of  $\la'$ into $\la_{}$ are \emph{linearly equivalent} if for each representation 
 $\pi: \la \rightarrow \mathfrak{gl}(V)$ the induced $\la'$-representations $\pi \circ \varphi$, $\pi \circ \varrho$ are equivalent, and we write
$\varphi \sim_L \varrho$.

\end{itemize}

We define linear equivalence of subalgebras much as we did for embeddings.
\begin{itemize}
\item  Two subalgebras $\la'$ and $\la''$ of $\la$ are linearly equivalent if for every representation $\pi:\la \rightarrow \mathfrak{gl}(V)$ the subalgebras $\rho(\la')$, $\rho(\la'')$ of $\mathfrak{gl}(V)$ are conjugate under $\mathrm{GL}(V)$.
Clearly if two embeddings or subalgebras are equivalent up to inner automorphism, then they are linearly equivalent.  But, the converse is not in general true.   

\end{itemize}

\section{The Levi decomposable subalgebras of $\sll$}\label{levi}

To determine the Levi decomposable  subalgebras of $\sll$, we first decompose $\sll$ with respect to the adjoint actions of $A^1_1$ and $A^2_1$, respectively:

\begin{equation}\label{acc3}
 \arraycolsep=1.4pt\def\arraystretch{1.3}
\begin{array}{ccccccccccccc}
\sll &\cong_{A^1_1}& \langle x_3, y_3, h_1+h_2\rangle&\oplus& \langle x_2, y_1 \rangle &\oplus&
\langle x_1, y_2 \rangle&\oplus& \\ &&\langle  h_1-h_2 \rangle  \\
&  \cong_{A^1_1} & V_{}(2) &\oplus&V_{}(1)&\oplus& V_{}(1) &\oplus&V_{}(0),
\end{array}
\end{equation}
\begin{equation}\label{ac3}
 \arraycolsep=1.4pt\def\arraystretch{1.3}
\begin{array}{ccccccccccccccccc}
\sll &\cong_{A^2_1}&\langle x_1+x_2,2 y_1+2y_2, 2h_1+2h_2   \rangle&\oplus&\\&&  \langle x_3, x_1-x_2,h_1-h_2,y_1-y_2, y_3\rangle \\
&  \cong_{A^2_1} & V_{}(2) &\oplus&V_{}(4),
\end{array}
\end{equation}
where $V(n)$ is the $n+1$ dimensional, irreducible representation of $A_1$.  In each decomposition $V(2)\cong A_1$ as a subalgebra of  $A_2$.  The remaining representations in the decompositions, or combinations thereof, 
give us the potential solvable components for Levi decomposable subalgebras.

\begin{lemma}\label{lem1}
Let  $\psi: \la \inplus \mathfrak{s} \rightarrow \mathfrak{h} \inplus \mathfrak{r}$ be a Lie algebra isomorphism, where  $\mathfrak{g}$  and  $\mathfrak{h}$  are semisimple and  $\mathfrak{s}$  and  $\mathfrak{r}$  are solvable.
Then $\psi(\mathfrak{s}) = \mathfrak{r}$.
\end{lemma}
\begin{proof}
Let $\pi: \mathfrak{h} \inplus \mathfrak{r} \rightarrow  \mathfrak{h}$ be the projection map of  $\mathfrak{h} \inplus \mathfrak{r}$ onto  $\mathfrak{h}$.  Then,
$\pi( \psi(\mathfrak{s}))$ is a solvable ideal of $\mathfrak{h}$.   Since  $\mathfrak{h}$ is semisimple, $\pi( \psi(\mathfrak{s}))=0$.  Hence,  $\psi(\mathfrak{s})\subseteq \mathfrak{r}$.  The Levi factor in a Levi decomposition is unique, up to isomorphism (i.e., $\la \cong\mathfrak{h}$).  Hence, dimension considerations imply $\psi(\mathfrak{s}) = \mathfrak{r}$.
\end{proof}

\begin{theorem}\label{les2}
A classification of Levi decomposable subalgebras of $A_2$, up to inner automorphism, is given by: 
\begin{equation}
 \arraycolsep=1.4pt\def\arraystretch{1.5}
\begin{array}{lllll}
\langle x_3, y_3, h_1+h_2\rangle \oplus \langle h_1-h_2\rangle &\cong& A_1 \oplus J,\\
 \langle x_3, y_3, h_1+h_2\rangle \inplus \langle x_1, y_2 \rangle &\cong& A_1 \inplus_{\varphi_1} K_1, \\
 \langle x_3, y_3, h_1+h_2\rangle \inplus \langle x_2,  y_1\rangle &\cong& A_1 \inplus_{\varphi_1} K_1,\\
 \langle x_3, y_3, h_1+h_2\rangle \inplus \langle h_1-h_2,x_1, y_2 \rangle &\cong& A_1 \inplus_{\varphi_2} L_{2},\\
 \langle x_3, y_3, h_1+h_2\rangle \inplus \langle h_1-h_2,x_2,  y_1\rangle &\cong& A_1 \inplus_{\varphi_2} L_{2},
\end{array}
\end{equation}
where the Lie algebra homomorphisms $\varphi_1: A_1 \rightarrow Der(K_1)$; and 
$\varphi_2: A_1 \rightarrow Der(L_{2})$ are defined in the proof below.
\end{theorem}
\begin{proof}
The decomposition of $A_2$ with respect to the adjoint action of $A^2_1$ in Eq. \eqref{ac3} contains exactly two components:
$\langle x_1+x_2,2 y_1+2y_2, 2h_1+2h_2   \rangle$ $\cong$ $A_1$ and $ \langle x_3, x_1-x_2,h_1-h_2,y_1-y_2, y_3\rangle $.  Thus, the only nontrivial extension of $A_1^2$ in $A_2$ is $A_2$ itself.  This implies that
$A_1^2$ cannot be extended to a Levi decomposable subalgebra of $A_2$.  We now turn our attention to $A_1^1$.

In the decomposition of $A_2$ with respect to the adjoint action of $A^1_1$ in Eq. \eqref{acc3}, the component $\langle x_3, y_3, h_1+h_2\rangle$ is 
isomorphic to $A_1$, and the other components, or combinations thereof, give us the potential extensions of $\langle x_3, y_3, h_1+h_2\rangle$$\cong$ $A_1$.

The component $\langle h_1-h_2 \rangle$ is one-dimensional, and is therefore a subalgebra of $A_2$.  Since it is the only one-dimensional component,
the only one-dimensional extension of $\langle x_3, y_3, h_1+h_2\rangle$, up to inner automorphism, is
\begin{equation}
\langle x_3, y_3, h_1+h_2\rangle \oplus \langle h_1-h_2\rangle.
\end{equation}
Note that the sum above is direct.  We have 
\begin{equation}
\begin{array}{lllll}
A_1 \oplus J &\cong& \langle x_3, y_3, h_1+h_2\rangle \oplus \langle h_1-h_2\rangle.
\end{array}
\end{equation}

We now consider $2$-dimensional extensions of $A_1^1$ in $A_2$.   The possible $2$-dimensional extensions of $A^1_1$ are by an
irreducible $A_1$-representation $V(1)$ with respect to the adjoint action of $A_1^1$. From Eq. \eqref{acc3}, a highest weight vector 
of such a representation is a nonzero linear combination of $x_1$ and $x_2$.  However, only (a nonzero multiple of) $x_1$ or $x_2$ generates a representation isomorphic to 
$V(1)$ with respect to the adjoint action of $A^1_1$ which is also a 
$2$-dimensional subalgebra of $A_2$.  

Hence, the possible $2$-dimensional extensions of $A_1^1$ 
are  by $\langle x_1, y_2 \rangle$, or  $\langle x_2, y_1 \rangle$.  Note that both  $\langle x_1, y_2 \rangle$, and  $\langle x_2, y_1 \rangle$ are abelian subalgebras, and hence isomorphic to $K_1$.  We thus have (at most) two $5$-dimensional Levi decomposable subalgebras  of $A_2$:
\begin{equation}\label{not1}
 \arraycolsep=1.4pt\def\arraystretch{1.5}
\begin{array}{lllll}
\langle x_3, y_3, h_1+h_2\rangle \inplus \langle x_1, y_2 \rangle,&
\langle x_3, y_3, h_1+h_2\rangle \inplus \langle x_2,  y_1\rangle.
\end{array}
\end{equation}
These subalgebras are isomorphic.  The Chevalley involution of $A_2$ induces an isomorphism between $\langle x_3, y_3, h_1+h_2\rangle \inplus \langle x_1, y_2 \rangle$ and $\langle x_3, y_3, h_1+h_2\rangle \inplus \langle x_2,  y_1\rangle$.

Define
\begin{equation}
\begin{array}{llll}
\chi: &A_1 &\hookrightarrow& A_2\\
 & x & \mapsto & x_3\\
  & y & \mapsto & y_3\\
   & h & \mapsto & h_1+h_2,
\end{array}
\end{equation}
where $\{x, y, h \}$ is a Chevalley basis of $A_1$, and 
\begin{equation}
\begin{array}{ccccccccccccccccc}
\psi_1: &K_1 &\rightarrow& K_1 &\subseteq A_2, \\
 & z_1 & \mapsto & x_1 \\
  & z_2 & \mapsto & y_2 \\
\end{array}
\end{equation}
\begin{equation}
\begin{array}{rllll}
\varphi_1: &A_1 &\rightarrow& Der(K_1)\\
\varphi_1(L): &z_i &\mapsto& \psi_1^{-1} ( [\chi(L), \psi_1(z_i)]), i=1,2,
\end{array}
\end{equation}
then
\begin{equation}
 \arraycolsep=1.4pt\def\arraystretch{1.5}
\begin{array}{lllll}
A_1 \inplus_{\varphi_1} K_1 &\cong& \langle x_3, y_3, h_1+h_2\rangle \inplus \langle x_1, y_2 \rangle,\\
A_1 \inplus_{\varphi_1} K_1 &\cong& \langle x_3, y_3, h_1+h_2\rangle \inplus \langle x_2,  y_1\rangle.
\end{array}
\end{equation}

We now show that  $\langle x_3, y_3, h_1+h_2\rangle \inplus \langle x_1, y_2 \rangle$ and $\langle x_3, y_3, h_1+h_2\rangle \inplus \langle x_2,  y_1\rangle$   are inequivalent subalgebras of $A_2$.  By way of contradiction, suppose that they are equivalent.  Then, there exists $A\in \mathrm{SL}(3,\mathbb{C})$ such that
$A^{-1} \langle x_3, y_3, h_1+h_2\rangle \inplus \langle x_1, y_2 \rangle A$ $=$ $\langle x_3, y_3, h_1+h_2\rangle \inplus \langle x_2,  y_1\rangle$.


By Lemma \ref{lem1}, $A^{-1} \langle x_1, y_2 \rangle A$ $=$ $\langle x_2,  y_1\rangle$.  However, no such $A\in \mathrm{SL}(3,\mathbb{C})$ exists, as one 
can show by direct calculation.  Hence, $\langle x_3, y_3, h_1+h_2\rangle \inplus \langle x_1, y_2 \rangle$ and $\langle x_3, y_3, h_1+h_2\rangle \inplus \langle x_2,  y_1\rangle$   are inequivalent subalgebras of $A_2$.


We now consider  $3$-dimensional extensions of $A_1^1$ in $A_2$.  By Eq. \eqref{acc3}, 
a $3$-dimensional extension of $A^1_1$ with respect to the adjoint action of $A^1_1$ must be isomorphic to $V(1)\oplus V(0)$.  The highest weight vector
of $V(0)$ must be (a nonzero scalar multiple of) $h_1-h_2$, and the highest weight vector
of $V(1)$ must be a nonzero linear combination of $x_1$ and $x_2$.  As above, it is only a scalar multiple of $x_1$ or $x_2$ that can contribute to a representation of 
$V(1)\oplus V(0)$ which is also a $3$-dimensional subalgebra.

Hence, the possible $3$-dimensional extensions of $A_1^1$ are  $\langle h_1-h_2, x_1, y_2 \rangle$, and $\langle h_1-h_2, x_2, y_1 \rangle$, both of which are isomorphic to $L_{3,0}$. 
It follows that we have (at most) two $6$-dimensional Levi decomposable subalgebras  of $A_2$:
\begin{equation}\label{not2}
 \arraycolsep=1.4pt\def\arraystretch{1.5}
\begin{array}{lllll}
\langle x_3, y_3, h_1+h_2\rangle \inplus \langle h_1-h_2,x_1, y_2 \rangle,\\
\langle x_3, y_3, h_1+h_2\rangle \inplus \langle h_1-h_2,x_2,  y_1\rangle.
\end{array}
\end{equation}
These subalgebras are isomorphic.  The Chevalley involution of $A_2$ induces an isomorphism between $\langle x_3, y_3, h_1+h_2\rangle \inplus \langle h_1-h_2, x_1, y_2 \rangle$ and $\langle x_3, y_3, h_1-h_2\rangle \inplus \langle h_1-h_2, x_2,  y_1\rangle$.

Define
\begin{equation}
 \arraycolsep=1.4pt\def\arraystretch{1.5}
\begin{array}{ccccccccccccccccc}
\psi_2: &L_{2} &\rightarrow& L_{2} &\subseteq A_2, \\ 
& z_1 & \mapsto &  x_1 \\
  & z_2 & \mapsto &  y_2\\
   & z_3 & \mapsto & \frac{1}{3}(h_1-h_2),
\end{array}
\end{equation}
\begin{equation}
\begin{array}{rllll}
\varphi_2: &A_1 &\rightarrow& Der(L_{2})\\
\varphi_2(L): &z_i &\mapsto& \psi_2^{-1} ( [\chi(L), \psi_2(z_i)]), i=1,2,3,
\end{array}
\end{equation}
then
\begin{equation}
 \arraycolsep=1.4pt\def\arraystretch{1.5}
\begin{array}{lllll}
A_1 \inplus_{\varphi_2} L_{2} &\cong& \langle x_3, y_3, h_1+h_2\rangle \inplus \langle h_1-h_2,x_1, y_2 \rangle,\\
A_1 \inplus_{\varphi_2} L_{2}&\cong& \langle x_3, y_3, h_1+h_2\rangle \inplus \langle h_1-h_2,x_2,  y_1\rangle.
\end{array}
\end{equation}

We now show that  $\langle x_3, y_3, h_1+h_2\rangle \inplus \langle h_1-h_2, x_1, y_2 \rangle$ and $\langle x_3, y_3, h_1+h_2\rangle \inplus \langle h_1-h_2, x_2, y_1 \rangle$   are inequivalent subalgebras of $A_2$.  By way of contradiction, suppose that they are equivalent.  Then, there exists $A\in \mathrm{SL}(3,\mathbb{C})$ such that
$A^{-1} \langle x_3, y_3, h_1+h_2\rangle \inplus \langle h_1-h_2, x_1, y_2 \rangle A$ $=$ $\langle x_3, y_3, h_1+h_2\rangle \inplus \langle h_1-h_2, x_2,  y_1\rangle$.


By Lemma \ref{lem1}, $A^{-1} \langle h_1-h_2, x_1, y_2 \rangle A$ $=$ $\langle h_1-h_2, x_2,  y_1\rangle$.  However, no such $A\in \mathrm{SL}(3,\mathbb{C})$ exists, as one 
can show by direct calculation.  Hence, $\langle x_3, y_3, h_1+h_2\rangle \inplus \langle h_1-h_2, x_1, y_2 \rangle$ and $\langle x_3, y_3, h_1+h_2\rangle \inplus \langle h_1-h_2, x_2,  y_1\rangle$   are inequivalent subalgebras of $A_2$.

Finally we consider $4$-dimensional extensions of $A^1_1$ in $A_2$.  From Eq. \eqref{acc3}, such  an extension must be isomorphic to 
$V(1)\oplus V(1)$ as a representation with respect to the adjoint action of $A^{1}_1$.   However, as a subalgebra of $A_2$, such a representation will generate a subalgebra of 
dimension greater than $4$.  Thus there are no $4$-dimensional extensions of $A^1_1$ in $A_2$.
\end{proof}

If we make the following definitions:
\begin{equation}
 \arraycolsep=1.4pt\def\arraystretch{1.5}
\begin{array}{lllll}
(A_1 \oplus_{} J)^1 &\equiv& \langle x_3, y_3, h_1+h_2\rangle \oplus \langle h_1-h_2\rangle,\\
(A_1 \inplus_{\varphi_1} K_1)^1 &\equiv& \langle x_3, y_3, h_1+h_2\rangle \inplus \langle x_1, y_2 \rangle,\\
(A_1 \inplus_{\varphi_1} K_1)^2&\equiv& \langle x_3, y_3, h_1+h_2\rangle \inplus \langle x_2,  y_1\rangle,\\
(A_1 \inplus_{\varphi_2} L_{2})^1 &\equiv& \langle x_3, y_3, h_1+h_2\rangle \inplus \langle h_1-h_2, x_1, y_2 \rangle,\\
(A_1 \inplus_{\varphi_2} L_{2})^2&\equiv& \langle x_3, y_3, h_1+h_2\rangle \inplus \langle h_1-h_2, x_2,  y_1\rangle,
\end{array}
\end{equation}
we get the concise summary of Levi decomposable subalgebras of $A_2$ in Table \ref{enddd} at the end of the paper.  
\begin{remark}
Note that 
 $\mathfrak{gl}(2,\mathbb{C})\cong A_1 \oplus_{} J,$ where $\mathfrak{gl}(2,\mathbb{C})$ is the general linear
 algebra of $2\times 2$ complex matrices.  By Theorem \ref{les2}, there is a  unique subalgebra of $A_2$ isomorphic to $\mathfrak{gl}(2,\mathbb{C})$.
\end{remark}

\section{The solvable subalgebras of $\sll$}\label{solvvv}

\subsection{Introduction}

We shall proceed by cases to classify the solvable subalgebras of $\sll$ for each isomorphism class.  

\subsection{One-dimensional subalgebras of $\sll$}

\begin{theorem}\label{hgggg}
A classification of $1$-dimensional (solvable) subalgebras of $\sll$, up to inner automorphism, is given by: 
\begin{equation}
 \arraycolsep=1.4pt\def\arraystretch{1.2}
\begin{array}{lllllllll}
J^1&=& \langle x_1+x_2 \rangle, \\
 J^{2}&=& \langle x_1 \rangle, \\  J^{3}&=& \langle h_1+2h_2+x_1 \rangle,\\
  J^{4,\alpha}&=& \langle  h_1+ \alpha h_2 \rangle,
\end{array} 
\end{equation}
for $\alpha \in \mathbb{C}$, where  $J^{4,\alpha} \sim  J^{4,\beta}$ if and only if $\alpha = 1-\beta$,  $\alpha =\beta$, 
$\alpha =(-1+\alpha)\beta$,  $\alpha =-(-1+\alpha)(-1+\beta)$,  $\alpha \beta=1$, or $-\alpha (-1+\beta)=1$. 
Equivalently,  $J^{4,\alpha} \sim  J^{4,\beta}$ if and only if 
$\beta = \alpha,  \frac{1}{\alpha}, 1 - \alpha,  \frac{1}{1 - \alpha},
\frac {\alpha}{-1+\alpha}$,  or $ \frac {-1+\alpha}{\alpha}$.
\end{theorem}
\begin{proof}
Every element of $A_2$ is equivalent to a traceless matrix in Jordan Normal Form, via conjugation by an element of $\mathrm{GL}(3,\mathbb{C})$, and hence by an 
element of $\mathrm{SL}(3,\mathbb{C})$. Thus, every element of $A_2$ is equivalent to a matrix in one of the following forms:
\begin{equation}\label{segg}
\begin{array}{llllllllll}
\left(
\begin{array}{ccc}
0 & 1 & 0   \\
 0 & 0 & 1   \\
 0& 0 & 0
\end{array}
\right), & \left(
\begin{array}{ccc}
\alpha & 1 & 0   \\
 0 & \alpha & 0   \\
 0& 0 &-2\alpha  
\end{array}
\right), & \text{or} \left(
\begin{array}{ccc}
\alpha & 0 & 0   \\
 0 & \beta & 0   \\
 0& 0 &-\alpha-\beta  
\end{array}
\right),
\end{array}
\end{equation}
for some $\alpha, \beta \in \mathbb{C}$.  Because multiplying by a nonzero scalar does not affect the block type of the 
Jordan Normal Form,  a $1$-dimensional subalgebra  
generated by one of the above matrices is inequivalent to any algebra generated by 
a matrix of either of the other two types.  

Corresponding to the first matrix in Eq. \eqref{segg}, define the subalgebra
\begin{equation}
J^1 \equiv  \langle x_1+ x_2\rangle.
\end{equation}
Any subalgebra generated by a matrix whose Jordan Canonical Form is the 
first matrix in Eq. \eqref{segg} is equivalent to $J^{1}$.

The subalgebra generated by a matrix whose Jordan Canonical Form is the second matrix in Eq. \eqref{segg} is equivalent to the subalgebra generated by 
\begin{equation}\label{bgn}
\begin{array}{llllllllll}
\left(
\begin{array}{ccc}
0 & 1 & 0   \\
 0 & 0 & 0   \\
 0& 0 & 0
\end{array}
\right), & \text{or} \left(
\begin{array}{ccc}
1 & \frac{1}{\alpha} & 0   \\
 0 & 1 & 0   \\
 0& 0 & -2
\end{array}
\right),
\end{array}
\end{equation}
for $\alpha \in \mathbb{C}^*$.  The second matrix in Eq. \eqref{bgn} is conjugate in $\mathrm{SL}(3,\mathbb{C})$ to 
\begin{equation}
\begin{array}{llllllllll}
\left(
\begin{array}{ccc}
1 & 1& 0   \\
 0 & 1 & 0   \\
 0& 0 & -2
\end{array}
\right).
\end{array}
\end{equation}
Hence, we have the following inequivalent $1$-dimensional subalgbras
\begin{equation}
\begin{array}{lll}
J^2\equiv \langle x_1 \rangle, & J^3\equiv \langle h_1+2h_2+x_1 \rangle.
\end{array}
\end{equation}
Finally, we consider  $1$-dimensional subalgebras generated by the third matrix in Eq. \eqref{segg}, which is diagonal. The $1$-dimensional subalgebra generated by this matrix is equivalent to the subalgebra generated by
\begin{equation}
\begin{array}{llllllllll}
\left(
\begin{array}{ccc}
1 & 0& 0   \\
 0 & -1+\alpha & 0   \\
 0& 0 & -\alpha
\end{array}
\right)=h_1+\alpha h_2,
\end{array}
\end{equation}
for some $\alpha \in \mathbb{C}$ (not necessarily the same $\alpha$ as in Eq. \eqref{segg}). Define
\begin{equation}
\begin{array}{lll}
J^{4, \alpha} \equiv \langle h_1+ \alpha h_2\rangle. 
\end{array}
\end{equation}

Suppose $J^{4, \alpha} \sim J^{4, \beta}$. Then $h_1+\alpha h_2$ is conjugate by an element of $\mathrm{SL}(3,\mathbb{C})$ to 
$\lambda(h_1+\beta h_2)$, for some $\lambda \in \mathbb{C}^*$.  This implies that  the diagonal matrix $h_1+\alpha h_2$ must be equal to the diagonal matrix $\lambda(h_1+\beta h_2)$ after permutation of diagonal entries.  That is, $\{ 1, -1+\alpha, -\alpha\} =\{\lambda, \lambda(-1+\beta),-\lambda \beta \}$.  We now proceed in cases according to the 
value of $\lambda$.

\vspace{2mm}

\noindent \underline{Case $1$}.  $\lambda=1$.  Then $-\alpha = -1+\beta$, or $-\alpha =-\beta$.  That is, 
$\alpha = 1-\beta$, or $\alpha =\beta$.

\vspace{2mm}

\noindent \underline{Case $2$}.  $\lambda= -1+\alpha$.  Then $-\alpha =-(-1+\alpha) \beta$, or $-\alpha=(-1+\alpha) (-1+\beta)$.  That is,
$\alpha =(-1+\alpha)\beta$, or $\alpha =-(-1+\alpha)(-1+\beta)$.  
 
\vspace{2mm}

\noindent \underline{Case $3$}.  $\lambda=-\alpha$.  Then $1=\alpha \beta$, or $-1+\alpha = \alpha \beta$.  That is,
$\alpha \beta=1$, or $-\alpha (-1+\beta)=1$.

\vspace{2mm}

We have shown that if $J^{4, \alpha} \sim J^{4, \beta}$, then $\alpha = 1-\beta$,  $\alpha =\beta$, 
$\alpha =(-1+\alpha)\beta$,  $\alpha =-(-1+\alpha)(-1+\beta)$,  $\alpha \beta=1$, or $-\alpha (-1+\beta)=1$.  Further, if one of these conditions holds then $J^{4, \alpha} \sim J^{4, \beta}$.
\end{proof}

\subsection{Two-dimensional subalgebras of $\sll$}

We consider each of the two $2$-dimensional (solvable) Lie algebras as subalgebras of $\sll$ separately in the theorems below.

\begin{theorem}\label{kone}
A classification of the subalgebras of $\sll$ isomorphic to $K_1$, up to inner automorphism, is given by:
\begin{equation}
 \arraycolsep=1.4pt\def\arraystretch{1.2}
\begin{array}{llll}
\vspace{1mm}
K^1_1 &=& \langle x_1+x_2,x_3 \rangle, \\
\vspace{1mm}
K^2_1 &=& \langle x_1, h_1+2h_2 \rangle, \\
\vspace{1mm}
K^{3}_1 &=& \langle x_1, x_3  \rangle,\\
\vspace{1mm}
K^{4}_1 &=& \langle x_1, y_2  \rangle,\\
\vspace{1mm}
K^5_1 &=& \langle h_1, h_2 \rangle.
\end{array}
\end{equation}
\end{theorem}
\begin{proof}
Let  $K$ be a subalgebra  which is isomorphic to $K_1$ with basis $z_1, z_2$ so that $[z_1, z_2]=0$. From Theorem \ref{hgggg}, we have the following cases.

\vspace{2mm}

 \noindent \underline{Case 1}. $\langle z_1 \rangle$ is conjugate in $\mathrm{SL}(3, \mathbb{C})$ to $\langle x_1+x_2 \rangle$.  After conjugation, we may assume $z_1=x_1+x_2$.   The commutation relation $[z_1,z_2]=0$ implies  $z_2=a(x_1+x_2)+bx_3$, for $a, b \in \mathbb{C}$.  Hence, $K$ is equivalent to $\langle x_1+x_2, x_3 \rangle$.  Define
 \begin{equation}
 K^1_1 \equiv \langle x_1+x_2, x_3 \rangle.
 \end{equation}

\vspace{2mm}

 \noindent \underline{Case 2}. $\langle z_1 \rangle$ is conjugate in $\mathrm{SL}(3,\mathbb{C})$ to $\langle x_1\rangle$.   After conjugation by an element of $\mathrm{SL}(3,\mathbb{C})$, we may assume $z_1=x_1$.   The commutation relation $[z_1,z_2]=0$ implies $z_2=\alpha(h_1+2h_2)+\beta x_1+\gamma x_3+\delta y_2$, for $\alpha, \beta, \gamma, \delta \in \mathbb{C}$.  This implies $K$ is equivalent to $\langle x_1, \alpha (h_1+2h_2)+\gamma x_3+\delta y_2\rangle$.

\vspace{2mm}

 \noindent \underline{Case 2.1}. $\alpha \neq 0$.  Then $z_1=x_1$, and $z_2=\alpha(h_1+2h_2)+\gamma x_3+\delta y_2$. Conjugation by
  \begin{equation}
  \left(
\begin{array}{ccc}
1 & 0 & \frac{\gamma}{3 \alpha}   \\
 0 & 1 & 0   \\
 0& \frac{ \delta}{3 \alpha } &1  
\end{array}
\right) \in \mathrm{SL}(3, \mathbb{C}),
\end{equation}
 fixes $z_1=x_1$ and 
 \begin{equation}
 A^{-1} z_2 A = \left(
\begin{array}{ccc}
\alpha & -\frac{\gamma \delta}{3\alpha} & 0  \\
 0 & \alpha & 0   \\
 0& 0 & -2\alpha  
\end{array}
\right).
 \end{equation}
Hence, $K$ is equivalent to $\langle x_1, h_1+2 h_2\rangle$.  Define
 \begin{equation}
 K^2_1 \equiv \langle x_1, h_1+2h_2 \rangle.
 \end{equation}
 
 \vspace{2mm}
 
   \noindent \underline{Case 2.2}. $\alpha =  0$, $\gamma \neq 0$, $\delta \neq 0$.   Then $z_1=x_1$, and $z_2=\gamma x_3+\delta y_2$. It follows that
  $K = \langle x_1, x_3+ \frac{\delta}{\gamma}y_2 \rangle$.  
  If 
    \begin{equation}
  \begin{array}{lllllll}
A=\left(
\begin{array}{ccc}
\lambda & 0 & 0   \\
 0 & \frac{1}{\lambda^2} & 0   \\
 0& 0 & \lambda  
\end{array}
\right)\in \mathrm{SL}(3, \mathbb{C}), ~\text{with}~\lambda^3 =\frac{\delta}{\gamma}, 
\end{array}
  \end{equation}
then   
$A^{-1} \langle x_1, x_3+ \frac{\delta}{\gamma} y_2\rangle A=  \langle x_1, x_3+  y_2\rangle$.
  Further,  $B^{-1} \langle x_1, x_3+  y_2\rangle  
  B
  = K^1_1$, where 
      \begin{equation}
  \begin{array}{lllllll}
  B
 =\left(
\begin{array}{ccc}
-1& 1 & 0   \\
 0 & 0 & 1   \\
 0& 1 & 0 
\end{array}
\right)\in \mathrm{SL}(3, \mathbb{C}).
\end{array}
  \end{equation}

\vspace{2mm}

    \noindent \underline{Case 2.3}. $\alpha =  0$, $\gamma \neq 0$, $\delta = 0$.   Then $z_1=x_1$, and $z_2=\gamma x_3$, and $K$ is equivalent to $\langle x_1, x_3 \rangle$. Define
     \begin{equation}
 K^3_1 \equiv \langle x_1, x_3 \rangle.
 \end{equation}

\vspace{2mm}
 
    \noindent \underline{Case 2.4}. $\alpha =  0$, $\gamma = 0$, $\delta \neq 0$.   Then $z_1=x_1$, and $z_2=\delta y_2$, and
    $K$ is equivalent to $\langle x_1, y_2 \rangle$.
    Define
     \begin{equation}
 K^4_1 \equiv \langle x_1, y_2 \rangle.
 \end{equation}
 
  \vspace{2mm}

 \noindent \underline{Case 3}. $\langle z_1 \rangle$ is conjugate in $\mathrm{SL}(3, \mathbb{C})$ to $\langle h_1+2h_2+x_1\rangle$, so we may assume $z_1=h_1+2h_2+x_1$.   The commutation relation $[z_1,z_2]=0$ implies $z_2=a(h_1 + 2h_2)+bx_1$, for $a, b \in \mathbb{C}$.  This implies that $K$ is equivalent to
 $\langle x_1, h_1+2h_2 \rangle=K^2_1$.
 
 \vspace{2mm}
 
   \noindent \underline{Case 4}. $\langle z_1 \rangle$ is conjugate in $\mathrm{SL}(3,\mathbb{C})$ to $\langle h_1+ \alpha h_2\rangle$, so we may assume $z_1=h_1+\alpha h_2$. 
 
 \vspace{2mm}
 
  \noindent \underline{Case 4.1}. $\alpha=2$.    The commutation relation  $[z_1,z_2]=0$ implies that 
   \begin{equation}
    \arraycolsep=1.4pt\def\arraystretch{1.3}
  \begin{array}{lllllll}
  z_2=\left(
\begin{array}{ccc}
\beta & \delta & 0   \\
 \epsilon & -\beta+\gamma & 0   \\
 0& 0 & -\gamma
\end{array}
\right), 
\end{array}
  \end{equation}
   for $\beta, \gamma, \delta, \epsilon \in \mathbb{C}$.  Let $A \in \mathrm{SL}(3,\mathbb{C})$ 
  such that $A^{-1} z_1A=z_1$, then
  \begin{equation}\label{Eq:matrix4.1}
  \begin{array}{lllllll}
A=\left(
\begin{array}{ccc}
a & b & 0   \\
 c & d & 0   \\
 0& 0 & \frac{1}{ad-bc}  
\end{array}
\right),
\end{array}
  \end{equation}
  for $a, b, c, d \in \mathbb{C}$ such that $ad-bc \neq 0$.   If $\begin{pmatrix} a & b \\ c&d\\ \end{pmatrix} \in \mathrm{GL}(2, \mathbb{C})$ is chosen so 
  that it conjugates the upper left $2\times 2$ block of
  $z_2$ to its Jordan Normal Form,  then $A$ in Equation (\ref{Eq:matrix4.1}) fixes $z_{1}$ and
  conjugates $z_{2}$ to 
    \begin{equation}
  \begin{array}{lllllll}
  \left(
\begin{array}{ccc}
\lambda & 1  & 0 \\
 0 & \lambda & 0 \\
 0&0& -2\lambda 
\end{array}
\right) &  \text{or} ~ \left(
\begin{array}{ccc}
\lambda & 0 & 0 \\
 0 & \mu &0 \\
 0&0& -\lambda -\mu
\end{array}
\right).
\end{array}
  \end{equation}
 In the former case, $K$ is equivalent to  $\langle x_1, h_1+2h_2\rangle=K^2_1$.    In the latter case, $K$ is equivalent to  $\langle h_1, h_2 \rangle$.  Define
   \begin{equation}
 K^5_1 \equiv \langle h_1, h_2 \rangle.
 \end{equation}
  
\vspace{2mm}

  \noindent \underline{Case 4.2}.  $\alpha=-1$. The commutation relation  $[z_1,z_2]=0$ implies that 
  \begin{equation}
  \begin{array}{lllllll}
  z_2=\left(
\begin{array}{ccc}
\beta & 0 & -\delta   \\
 0 & -\beta+\gamma & 0   \\
 -\epsilon & 0 & -\gamma
\end{array}
\right), 
\end{array}
  \end{equation}
   for $\beta, \gamma, \delta, \epsilon \in \mathbb{C}$.  Let $A \in \mathrm{SL}(2,\mathbb{C})$ 
  such that $A^{-1} z_1A=z_1$, then
  \begin{equation}\label{Eq:matrix4.2}
  \begin{array}{lllllll}
A=\left(
\begin{array}{ccc}
a & 0 & b   \\
 0 & \frac{1}{ad-bc} & 0   \\
 c& 0 & d  
\end{array}
\right),
\end{array}
  \end{equation}
  for $a, b, c, d \in \mathbb{C}$ such that $ad-bc \neq 0$.    If $\begin{pmatrix} a & b \\ c&d\\ \end{pmatrix} \in \mathrm{GL}(2, \mathbb{C})$ is chosen so 
  that it conjugates the  $2\times 2$ matrix consisting of the outer corners of 
  $z_2$ to its Jordan Normal Form,  then $A$ in Equation (\ref{Eq:matrix4.2}) 
  fixes $z_{1}$ and 
  conjugates $z_{2}$ to 
    \begin{equation}
  \begin{array}{lllllll}
  \left(
\begin{array}{ccc}
\lambda & 0 &1   \\
 0 & -2 \lambda & 0\\
0 & 0 &\lambda  
\end{array}
\right) &   \text{or} ~ \left(
\begin{array}{ccc}
\lambda & 0 &0   \\
 0 & - \lambda-\mu & 0\\
0 & 0 &\mu  
\end{array}
\right).
\end{array}
  \end{equation}
  In the latter case, $K$ is equivalent to $\langle h_1, h_2 \rangle=K^5_1$.
 In the former case, $K$ is equivalent to  $\langle x_3, h_1-h_2\rangle$, which is equivalent to $K^2_1$ via conjugation by
   \begin{equation}
  \begin{array}{lllllll}
\left(
\begin{array}{ccc}
1  & 0 & 0   \\
 0 & 0& 1   \\
 0& -1 & 0  
\end{array}
\right) \in \mathrm{SL}(3, \mathbb{C}).
\end{array}
  \end{equation}
 
 \vspace{2mm}

  \noindent \underline{Case 4.3}.  $\alpha=\frac{1}{2}$.  The commutation relation  $[z_1,z_2]=0$ implies that 
  \begin{equation}
  \begin{array}{lllllll}
  z_2=\left(
\begin{array}{ccc}
-\gamma & 0 & 0   \\
 0 & \beta & \delta   \\
 0 & \epsilon & -\beta+\gamma
\end{array}
\right), 
\end{array}
  \end{equation}
   for $\beta, \gamma, \delta, \epsilon \in \mathbb{C}$.  Let $A \in \mathrm{SL}(2,\mathbb{C})$ 
  such that $A^{-1} z_1A=z_1$, then
 \begin{equation}
  \begin{array}{lllllll}
A=\left(
\begin{array}{ccc}
\frac{1}{ad-bc}  & 0 & 0   \\
 0 & a& b   \\
 0& c & d  
\end{array}
\right),
\end{array}
  \end{equation}
  for $a, b, c, d \in \mathbb{C}$ such that $ad-bc \neq 0$.  For appropriate choices of $a, b, c, d$ the matrix $A$ fixes $z_1$ and conjugates $z_2$ to
    \begin{equation}
  \begin{array}{lllllll}
  \left(
\begin{array}{ccc}
-2\lambda & 0 &0   \\
 0 &  \lambda & 1\\
0 & 0 &\lambda  
\end{array}
\right) &   \text{or}~ \left(
\begin{array}{ccc}
-\lambda-\mu & 0 &0   \\
 0 & \lambda & 0\\
0 & 0 &\mu  
\end{array}
\right).
\end{array}
  \end{equation}
  In the latter case, $K$ is equivalent to $\langle h_1, h_2 \rangle=K^5_1$.
  In the former case, $K$ is equivalent to $\langle x_2, h_1+\frac{1}{2}h_2\rangle$.    Then $B^{-1} \langle x_2, h_1+\frac{1}{2}h_2\rangle B$ $=$
 $\langle x_1, h_1+2h_2\rangle=K^2_1$, where
  \begin{equation}
  \begin{array}{lllllll}
B=\left(
\begin{array}{ccc}
0  & 0 & 1   \\
 1 & 0& 0   \\
 0& 1 & 0  
\end{array}
\right) \in \mathrm{SL}(3, \mathbb{C}).
\end{array}
  \end{equation}
 
\vspace{2mm}

  \noindent \underline{Case 4.4}.  $\alpha \neq 2, \frac{1}{2}, -1$.  The commutation relation  $[z_1,z_2]=0$ implies that 
  $z_2$ is diagonal.  Hence $K$ is equivalent to $\langle h_1, h_2\rangle=K^5_1$.
 
\vspace{2mm}

The above cases establish that each subalgebra of $A_2$ isomorphic to $K_1$ is equivalent to $K^1_1, K^2_1, K^3_1, K^4_1$, or $K^5_1$. Direct linear algebraic methods establish that these five subalgebras are pairwise inequivalent.  For instance, $K^5_1$ is not equivalent to any of the other subalgebras since an equivalence would imply $x_1$ or $x_1+x_2$ is diagonalizable, which of course is not true.  
\end{proof}

\begin{theorem}\label{ktwo}
A classification of the  subalgebras of $\sll$ isomorphic to $K_2$, up to inner automorphism, is given by:
\begin{equation}
 \arraycolsep=1.4pt\def\arraystretch{1.2}
\begin{array}{llll}
\vspace{1mm}
K_2^1 &=& \langle x_1+x_2, h_1+ h_2\rangle,\\
\vspace{1mm}
K_2^{2}&=& \langle x_1, -\frac{1}{3}h_2+\frac{1}{3}h_2+x_3 \rangle, \\
\vspace{1mm}
K_2^{3} &=& \langle x_1, -\frac{2}{3}h_1-\frac{1}{3}h_2+y_2\rangle,\\
K_2^{4,\alpha}&=& \langle x_1, \alpha h_1+(2\alpha +1)h_2\rangle,
\end{array}
\end{equation}
where $K_2^{4,\alpha} \sim K_2^{4,\beta}$ if and only if $\alpha=\beta$.
\end{theorem}
\begin{proof}
Let $K$ be a subalgebra of $\sll$ isomorphic to $K_2$ and let $z_1, z_2 \in K$ be nonzero such that $[z_1,z_2]=z_1$.  After conjugation in $\mathrm{SL}(3,\mathbb{C})$, we
may assume $z_1 \in J^1$, $J^2$, $J^3$, or $J^{4,\alpha}$, for some $\alpha \in \mathbb{C}$.

\vspace{2mm}

\noindent \underline{Case 1}.  $z_1 \in J^1 =\langle x_1+x_2\rangle$.  After scalar multiplication, we may 
assume $z_1=x_1+x_2$.  The commutation relation $[z_1,z_2]=z_1$ implies
$z_2= -h_1-h_2+ a x_3+b (x_1+x_2)$, for some $a, b \in \mathbb{C}$.
Hence, $K =\langle x_1+x_2, -h_1-h_2+ a x_3\rangle$.  After conjugation by 
\begin{equation}
\begin{array}{llll}
  \left(
\begin{array}{ccc}
1 & 0 & -\frac{a}{2}   \\
 0&  1 &0   \\
 0& 0 & 1
\end{array}
\right)\in \mathrm{SL}(3,\mathbb{C}),
\end{array}
\end{equation}
we have $K\sim \langle x_1+x_2, h_1+h_2\rangle \equiv K^1_2$.

\vspace{2mm}

\noindent \underline{Case 2}.  $z_1 \in J^2 =\langle x_1\rangle$.  After scalar multiplication, we may assume
$z_1= x_1$.  The commutation relation $[z_1,z_2]=z_1$ implies
$z_2= ah_1+(2a+1)h_2+bx_1+cx_3+dy_2$, for some $a, b, c, d \in \mathbb{C}$.   This implies
$K =\langle x_1, ah_1+(2a+1)h_2+cx_3+dy_2\rangle$.

\vspace{2mm}

\noindent \underline{Case 2.1}.  $a\neq -\frac{1}{3}, -\frac{2}{3}$.  Conjugation by
\begin{equation}
 \left(
\begin{array}{ccc}
1 & -\frac{cd}{3a+2} & \frac{c}{3a+1}   \\
 0 & 1 & 0   \\
 0& \frac{d}{3a+2} &1  
\end{array}
\right) \in \mathrm{SL}(3,\mathbb{C})
\end{equation}
yields $\langle x_1, a h_1+(2a+1)h_2\rangle$.

\vspace{2mm}

\noindent \underline{Case 2.2.1}.  $a= -\frac{1}{3}$, $c\neq 0$.  Conjugation by
\begin{equation}
 \left(
\begin{array}{ccc}
\lambda & -cd\lambda  & 0   \\
 0 &\lambda & 0   \\
 0& d\lambda &\frac{1}{\lambda^2}  
\end{array}
\right) \in \mathrm{SL}(3,\mathbb{C}),
\end{equation}
for $\lambda \in \mathbb{C}$ such that $\lambda^3=c$, yields $\langle x_1, -\frac{1}{3}h_1+\frac{1}{3}h_2+x_3\rangle$. 

\vspace{2mm}

\noindent \underline{Case 2.2.2}.  $a= -\frac{1}{3}$, $c=0$.  Conjugation by
\begin{equation}
 \left(
\begin{array}{ccc}
1 & 0 & 0   \\
 0 & 1 & 0   \\
 0& d &1  
\end{array}
\right) \in \mathrm{SL}(3,\mathbb{C})
\end{equation}
yields $\langle x_1, -\frac{1}{3}h_1+\frac{1}{3}h_2\rangle=\langle x_1, h_1-h_2\rangle$.

\vspace{2mm}

\noindent \underline{Case 2.3.1}.  $a= -\frac{2}{3}$, $d\neq0$.  Conjugation  by
\begin{equation}
 \left(
\begin{array}{ccc}
\lambda & -\frac{c(\lambda^2-1)}{\lambda^2} & -c+\frac{c(e^2-1)}{e^2}   \\
 0 & \lambda & 0   \\
 0& 1 &\frac{1}{\lambda^2}  
\end{array}
\right) \in \mathrm{SL}(3,\mathbb{C}),
\end{equation}
where $\lambda \in \mathbb{C}$ such that $d \lambda^3=1$,
yields $\langle x_1, -\frac{2}{3}h_1-\frac{1}{3}h_2+y_2\rangle$.

\vspace{2mm}

\noindent \underline{Case 2.3.2}.  $a= -\frac{2}{3}$, $d=0$.  Conjugation  by
\begin{equation}
 \left(
\begin{array}{ccc}
1 & 0 &-c   \\
 0 & 1 & 0   \\
 0& 0 &1  
\end{array}
\right) \in \mathrm{SL}(3,\mathbb{C})
\end{equation}
yields $\langle x_1, -\frac{2}{3}h_1-\frac{1}{3}h_2\rangle =\langle x_1, 2h_1+h_2\rangle$.

\vspace{2mm}

\noindent \underline{Case 3}.  $z_1 \in J^3 =\langle h_1+2h_2+x_1\rangle$.  After scalar multiplication, we may assume
$z_1= h_1+2h_2+x_1$.  The commutation relation $[z_1,z_2]=z_1$ has no solution, implying  that this case does not yield 
a subalgebra isomorphic to $K_2$.

\vspace{2mm}

\noindent \underline{Case 4}.  $z_1 \in J^{4, \alpha}=\langle h_1+\alpha h_2\rangle$.  After scalar multiplication, we may assume
$z_1= h_1+\alpha h_2$.  The commutation relation $[z_1,z_2]=z_1$ has no solution, implying that  this case does not yield 
a subalgebra isomorphic to $K_2$.

By cases $1$-$4$, each subalgebra of $A_2$ isomorphic to $K_2$ is equivalent to one of the following subalgebras:
\begin{equation}
 \arraycolsep=1.4pt\def\arraystretch{1.5}
\begin{array}{llllllllllllll}
K^1_2 &\equiv& \langle x_1+x_2, h_1+h_2\rangle,\\
K^2_2&\equiv& \langle x_1, -\frac{1}{3}h_1+\frac{1}{3}h_2+x_3\rangle,\\
K^3_2&\equiv& \langle x_1, -\frac{2}{3}h_1-\frac{1}{3}h_2+y_2\rangle,\\
K^{4,\alpha}_2&\equiv& \langle x_1, \alpha h_1+(2\alpha+1) h_2\rangle,\\
\end{array}
\end{equation}
for some $\alpha \in \mathbb{C}$.  

If $K^1_2$ were equivalent to $K^2_2$, then there would exist $A\in \mathrm{SL}(3, \mathbb{C})$ such that
$A^{-1} K^1_2 A= K^2_2$.  However, no such $A$ exists, as can be check by direct
computation.  Similarly,  $K^3_2$ and $K^{4, \alpha}_2$ are inequivalent to $K^1_2$; 
$K^2_2$ and $K^3_2$ are inequivalent; $K^2_2$ and $K^{4,\alpha}$   are inequivalent; and 
 $K^3_2$ and $K^{4,\alpha}$   are inequivalent.  Direct computation establishes that $K^{4,\alpha}_2$ is conjugate to $K^{4,\beta}_2$ if and only if
$\alpha=\beta$. \end{proof}

\subsection{Three-dimensional, solvable subalgebras of $\sll$}

In this subsection we use the following lemma from \cite{degraafa} (Lemma $2.1$):
\begin{lemma}\label{graaff}
Let $T$  be a solvable Lie algebra.  Then there is a subalgebra $S \subset T$ of codimension $1$, and $z\in T$  such that
$T = \langle z \rangle \inplus S$. 
\end{lemma}

\begin{theorem}\label{threeds}
A classification of the  $3$-dimensional solvable subalgebras of $\sll$, up to inner automorphism, is given by:
 \begin{equation}
 \arraycolsep=1.4pt\def\arraystretch{1.2}
\begin{array}{llllllll}
L_{2}^1 &=& \langle x_1, x_3, 2h_1+h_2  \rangle,\\
L_{2}^2 &=& \langle x_1, y_2, h_1-h_2 \rangle,\\
L_{3, -\frac{1}{4}}^1 &=& \langle x_1, x_3, 2h_1+h_2+x_2 \rangle,\\
L^2_{3, -\frac{1}{4}} &=& \langle y_1, y_3, 2h_1+ h_2+x_2  \rangle,\\
L^{1}_{3,-\frac{2}{9}}&=& \langle x_1+x_2, x_3, h_1+h_2\rangle,\\
L^{1}_{3,0}&=& \langle x_1, h_1, h_2\rangle,\\
L^{1, \alpha}_{3,-\frac{(2\alpha-1)(\alpha-2)}{9(\alpha-1)^2}}&=& \langle x_1, x_3, (\alpha-1)h_1+ \alpha h_2 \rangle,\\
 L^{2, \alpha}_{3,-\frac{(2\alpha-1)(\alpha-2)}{9(\alpha-1)^2}}&=& \langle x_1, y_2, h_1+ \alpha h_2 \rangle,\\
L_{4}^1 &=& \langle x_1, x_3, h_2 \rangle,\\
L_{4}^2 &=& \langle x_1, y_2, h_1+h_2 \rangle,\\
L^{1}_{5}&=& \langle x_1, x_2, x_3\rangle,\\
\end{array}
\end{equation}
for $\alpha \neq \pm 1$. The subscripts correspond to those of the appropriate 
isomorphism type in Equation (\ref{Eq:deg3}). Further, $L^{1, \alpha}_{3,-\frac{(2\alpha-1)(\alpha-2)}{9(\alpha-1)^2}}$ (resp. $L^{2, \alpha}_{3,-\frac{(2\alpha-1)(\alpha-2)}{9(\alpha-1)^2}}$) is 
equivalent to $L^{1, \beta}_{3,-\frac{(2\beta-1)(\beta-2)}{9(\beta-1)^2}}$ (resp. $L^{2, \beta}_{3,-\frac{(2\beta-1)(\beta-2)}{9(\beta-1)^2}}$) if and only if $\alpha=\beta$ or $\alpha \beta=1$.
\end{theorem}
\begin{proof}
Let $L$ be a solvable, $3$-dimensional subalgebra of $\sll$.  Then, by Lemma \ref{graaff}, 
\begin{equation}
L=\langle z \rangle \inplus K,
\end{equation}
where $K$ is a $2$-dimensional (solvable) subalgebra, and $z\in L$.  By Theorems \ref{kone} and \ref{ktwo}, after conjugation by an element of $\mathrm{SL}(3, \mathbb{C})$, we have the following cases.

\vspace{2mm}

\noindent \underline{Case $1$}. $L=\langle z \rangle \inplus K^1_1$ $=$ $\langle z \rangle \inplus \langle x_1+x_2, x_3 \rangle$.  The conditions $[z, x_1+x_2] \subseteq   \langle x_1+x_2, x_3\rangle$, and  $[z,  x_3] \subseteq   \langle x_1+x_2, x_3\rangle$ imply
that $z=a( h_1+h_2)+bx_1+cx_2+dx_3$, for some $a, b, c, d \in \mathbb{C}$, not all zero.   The values of $a, b,$ and $c$ yield the following two possibilities:
\begin{equation}
 \arraycolsep=1.4pt\def\arraystretch{1.5}
\begin{array}{llllll}
L\sim \langle x_1+x_2, x_3, h_1+h_2 \rangle, &\text{or}~ 
\langle x_1, x_2, x_3\rangle.
\end{array}
\end{equation}
The subalgebra $\langle x_1+x_2, x_3, h_1+h_2 \rangle$ is isomorphic to $L_{3, -\frac{2}{9}}$  via the isomorphism given by
\begin{equation}
 \arraycolsep=1.4pt\def\arraystretch{1.5}
\begin{array}{llllll}
z_1 &\mapsto & x_1+x_2+x_3,\\
z_2 &\mapsto &  \frac{1}{3}(x_1+x_2)+\frac{2}{3}x_3,\\
z_3 &\mapsto & \frac{1}{3}(h_1+h_2).
\end{array}
\end{equation}
The subalgebra $\langle x_1, x_2, x_3  \rangle$ is isomorphic to $L_5$ via the isomorphism given by
\begin{equation}
 \arraycolsep=1.4pt\def\arraystretch{1.5}
\begin{array}{llllll}
z_1 &\mapsto & x_1,\\
z_2 &\mapsto & x_3,\\
z_3 &\mapsto & x_2.
\end{array}
\end{equation}
Define
\begin{equation}
\begin{array}{llllll}
L_{3, -\frac{2}{9}}^1 &\equiv& \langle x_1+x_2, x_3, h_1+h_2 \rangle, \\
L_5^1 &\equiv& \langle x_1, x_2, x_3  \rangle.
\end{array}
\end{equation}

\vspace{2mm}

\noindent \underline{Case $2$}. $L=\langle z \rangle \inplus K^2_1$ $=$ $\langle z \rangle \inplus \langle x_1, h_1+2h_2 \rangle$.  The conditions $[z, x_1] \subseteq   \langle x_1, h_1+2h_2\rangle$, and  $[z,  h_1+2h_2] \subseteq   \langle x_1, h_1+2h_2\rangle$ imply
that $z=a h_1+ bh_2+cx_1$, for some $a, b, c \in \mathbb{C}$, not all zero.   Then, $L =\langle x_1, h_1, h_2\rangle$.  The subalgebra $\langle x_1, h_1, h_2  \rangle$ is isomorphic to $L_{3, 0}$ via the isomorphism given by
\begin{equation}
 \arraycolsep=1.4pt\def\arraystretch{1.5}
\begin{array}{llllll}
z_1 &\mapsto & h_1+2h_2+x_1,\\
z_2 &\mapsto & x_1,\\
z_3 &\mapsto & \frac{1}{2}h_1+x_1.
\end{array}
\end{equation}
Define
\begin{equation}
\begin{array}{llllll}
L_{3, 0}^1 &\equiv& \langle x_1, h_1, h_2  \rangle.
\end{array}
\end{equation}

\vspace{2mm}

\noindent \underline{Case $3$}. $L=\langle z \rangle \inplus K^3_1$ $=$ $\langle z \rangle \inplus \langle x_1, x_3 \rangle$.  
The conditions
$[z, x_1] \subseteq   \langle x_1, x_3 \rangle$, and  $[z,  x_3] \subseteq   \langle x_1, x_3\rangle$ imply
that $L =\langle  x_1, x_3, ax_2+ by_2+ ch_1+dh_2\rangle$, $a, b, c, d \in \mathbb{C}$,  not all zero. 
In matrix form, we have
\begin{equation}
ax_2+ by_2+ ch_1+dh_2=  \left(
\begin{array}{ccc}
c & 0 & 0   \\
 0 & -c+d & a   \\
 0& b & -d  
\end{array}
\right).
\end{equation}
We now consider the cases $4ab+c^2-4cd+4d^2 \neq 0$, and 
$4ab+c^2-4cd+4d^2 = 0$.

\vspace{2mm}

\noindent \underline{Case $3.1$}.  $4ab+c^2-4cd+4d^2 \neq 0$.  In this case $ax_2+ by_2+ ch_1+dh_2$ is diagonalizable, as is its lower right $2\times 2$ block matrix, of course.
For any block matrix 
\begin{equation}
A= \left(
\begin{array}{ccc}
\frac{1}{\det(G)} & 0    \\
 0 &  G
\end{array}
\right) \in \mathrm{SL}(3, \mathbb{C}),
\end{equation}
where $G\in \mathrm{GL}(2, \mathbb{C})$,
$A^{-1} \langle x_1, x_3\rangle A=  \langle x_1, x_3\rangle $.  We may choose such a block matrix $A$  which diagonalizes the lower right $2\times 2$ block of  $ax_2+ by_2+ ch_1+dh_2$ such that  $A^{-1} \langle ax_2+ by_2+ ch_1+dh_2 \rangle A= \langle (\alpha-1) h_1+\alpha h_2\rangle$, for some $\alpha \in \mathbb{C}$.  We now proceed in cases based on the value of
$\alpha$.

\vspace{2mm}

\noindent \underline{Case $3.1.1$}.  $\alpha=-1$. Then $L$ is equivalent to $\langle x_1, x_3, 2h_1+h_2\rangle$, which is isomorphic to $L_2$ via the isomorphism given by
\begin{equation}
 \arraycolsep=1.4pt\def\arraystretch{1.5}
\begin{array}{llllll}
z_1 &\mapsto & x_1,\\
z_2 &\mapsto & x_3,\\
z_3 &\mapsto & \frac{1}{3}(2h_1+h_2).
\end{array}
\end{equation}
Define
\begin{equation}
\begin{array}{llllll}
L_{2}^1 &\equiv& \langle x_1, x_3, 2h_1+h_2  \rangle.
\end{array}
\end{equation}

\vspace{2mm}

\noindent \underline{Case $3.1.2$}.  $\alpha=1$. Then $L$ is equivalent to $\langle x_1, x_3, h_2\rangle$, which is isomorphic to $L_{4}$ via the isomorphism given by
\begin{equation}
 \arraycolsep=1.4pt\def\arraystretch{1.5}
\begin{array}{llllll}
z_1 &\mapsto & x_1+x_3,\\
z_2 &\mapsto & x_1-x_3,\\
z_3 &\mapsto & -h_2.
\end{array}
\end{equation}
Define
\begin{equation}
\begin{array}{llllll}
L_{4}^1 &\equiv& \langle x_1, x_3, h_2 \rangle.
\end{array}
\end{equation}

\vspace{2mm}

\noindent \underline{Case $3.1.3$}.  $\alpha\neq  \pm 1$.   The subalgebra $\langle x_1, x_3, (\alpha-1)h_1+\alpha h_2  \rangle$ is isomorphic to $L_{3, -\frac{(2\alpha-1)(\alpha-2)}{9(\alpha-1)^2}}$ via the isomorphism given by
\begin{equation}
 \arraycolsep=1.4pt\def\arraystretch{1.5}
\begin{array}{llllll}
z_1 &\mapsto &x_1+x_3,\\
z_2 &\mapsto & \frac{\alpha-2}{3 (\alpha-1)}x_1+\frac{2\alpha-1}{3(\alpha-1)} x_3,\\
z_3 &\mapsto &  \frac{1}{3 (\alpha-1)}( (\alpha-1)h_1+\alpha h_2).
\end{array}
\end{equation}
For $\alpha \neq \pm 1$, define
\begin{equation}
L_{3, -\frac{(2\alpha-1)(\alpha-2)}{9(\alpha-1)^2}}^{1, \alpha} \equiv \langle x_1, x_3, (\alpha-1)h_1+\alpha h_2  \rangle.
\end{equation}

We now establish that $L_{3, -\frac{(2\alpha-1)(\alpha-2)}{9(\alpha-1)^2}}^{1, \alpha}$ is equivalent to $L_{3, -\frac{(2\beta-1)(\beta-2)}{9(\beta-1)^2}}^{1, \beta}$ if and only if
$\alpha =\beta$ or $\alpha \beta=1$.  First suppose that $L_{3, -\frac{(2\alpha-1)(\alpha-2)}{9(\alpha-1)^2}}^{1, \alpha}$ is equivalent to $L_{3, -\frac{(2\beta-1)(\beta-2)}{9(\beta-1)^2}}^{1, \beta}$, then these subalgebras are  isomorphic so that   $-\frac{(2\alpha-1)(\alpha-2)}{9(\alpha-1)^2}=-\frac{(2\beta-1)(\beta-2)}{9(\beta-1)^2}$. Then
\begin{equation}\label{wowg}
 \arraycolsep=1.4pt\def\arraystretch{1.3}
\begin{array}{lllll}
(2\alpha-1)(\alpha-2)(\beta-1)^2=(2\beta-1)(\beta-2)(\alpha-1)^2.
\end{array}
\end{equation}
Expanding, and simplifying Eq. \eqref{wowg}  yields
\begin{equation}
 \arraycolsep=1.4pt\def\arraystretch{1.3}
\begin{array}{lllll}
&\alpha^2\beta-\alpha \beta^2-\alpha+\beta=0 \\
\Rightarrow & (\alpha \beta -1)(\alpha-\beta)=0\\
\Rightarrow & \alpha \beta =1 ~\text{or}~\alpha=\beta.
\end{array}
\end{equation}

Now suppose that $\alpha \beta =1$, or $\alpha=\beta$, then $-\frac{(2\alpha-1)(\alpha-2)}{9(\alpha-1)^2}=-\frac{(2\beta-1)(\beta-2)}{9(\beta-1)^2}$ so that $L_{3, -\frac{(2\alpha-1)(\alpha-2)}{9(\alpha-1)^2}}^{1, \alpha}$ is isomorphic to $L_{3, -\frac{(2\beta-1)(\beta-2)}{9(\beta-1)^2}}^{1, \beta}$.  Further, they are equivalent via conjugation by
\begin{equation}
 \left(
\begin{array}{ccc}
1 & 0 &0   \\
 0 & 0 & 1   \\
 0& -1 &0  
\end{array}
\right) \in \mathrm{SL}(3,\mathbb{C}),
\end{equation}
in the case that $\alpha \beta=1$.

We now  establish that $L_{3, -\frac{(2\alpha-1)(\alpha-2)}{9(\alpha-1)^2}}^{1, \alpha}$ is not equivalent to any subalgebra defined
previously in this proof.  In particular, we must consider $\alpha\neq \pm 1$ for which
$-\frac{(2\alpha-1)(\alpha-2)}{9(\alpha-1)^2}=-\frac{2}{9}$, or $0$.  That is, we must
consider $\alpha =0, \frac{1}{2},$ and $2$. But, we may exclude the $\alpha=\frac{1}{2}$ case since $L_{3,0}^{1, \frac{1}{2}}$ is  equivalent to 
$L_{3,0}^{1, 2}$.

If $\alpha=0$, then $L_{3, -\frac{(2\alpha-1)(\alpha-2)}{9(\alpha-1)^2}}^{1, \alpha}=L_{3, -\frac{2}{9}}^{1, 0}$.  Straightforward calculation shows that
$L_{3, -\frac{2}{9}}^{1, 0}$ is inequivalent to $L_{3, -\frac{2}{9}}^{1}$.

If $\alpha=2$, then $L_{3, -\frac{(2\alpha-1)(\alpha-2)}{9(\alpha-1)^2}}^{1, \alpha}=L_{3,0}^{1, 2}$.  Straightforward calculation shows that
$L_{3, 0}^{1, 2}$ is inequivalent to $L_{3, 0}^{1}$.

\vspace{2mm}

\noindent \underline{Case $3.2$}.  $4ab+c^2-4cd+4d^2 = 0$.  Then $d=\frac{1}{2}c\pm\sqrt{-ab}$, so that
\begin{equation}
L=\bigg\langle x_1, x_3, ax_2+by_2+ch_1+ \bigg(\frac{1}{2}c\pm\sqrt{-ab}\bigg) h_2\bigg\rangle.
\end{equation}

\vspace{2mm}

\noindent \underline{Case $3.2.1$}.  $a\neq 0$, $b\neq 0$, $c\neq 0$.   Then,  conjugation by 
\begin{equation}
A=\frac{1}{\sqrt[3]{-b}} \left(
\begin{array}{ccc}
1 & 0 &0   \\
 0 & \pm \sqrt{-ab} & 1   \\
 0& b &0  
\end{array}
\right) \in \mathrm{SL}(3,\mathbb{C}),
\end{equation}
yields $L \sim \langle x_1, x_3, ch_1+\frac{c}{2}h_2+x_2\rangle$. 
After an additional appropriate conjugation we get $L \sim \langle x_1, x_3, 2h_1+h_2+x_2\rangle$, which is isomorphic to $L_{3, -\frac{1}{4}}$  via the isomorphism given by 
\begin{equation}
 \arraycolsep=1.4pt\def\arraystretch{1.5}
\begin{array}{llllll}
z_1 &\mapsto & x_1+x_3,\\
z_2 &\mapsto & \frac{1}{2} x_1+\frac{2}{3}x_3,\\
z_3 &\mapsto & \frac{1}{6}(2h_1+h_2+x_2).
\end{array}
\end{equation}
Define
\begin{equation}
\begin{array}{llllll}
L_{3, -\frac{1}{4}}^1 &\equiv& \langle x_1, x_3, 2h_1+h_2+x_2 \rangle.
\end{array}
\end{equation}
Note that $L_{3, -\frac{1}{4}}$ is not isomorphic--let alone equivalent--to any subalgebra defined previously in this proof.  In particular, 
 $L_{3, -\frac{1}{4}} \ncong L_{3, -\frac{(2\alpha-1)(\alpha-2)}{9(\alpha-1)^2}}$ for all $\alpha \neq \pm 1$.

\vspace{2mm}

\noindent \underline{Case $3.2.2$}.  At least one of $a=0$, $b= 0$, or $c= 0$.  If precisely two of $a, b, c$ are zero, then $L$ is equal to 
$\langle x_1, x_3, 2h_1+h_2\rangle=L^1_2$, $\langle x_1, x_2, x_3\rangle=L^1_5$, or  $\langle x_1, x_3, y_2 \rangle\sim L_5^1$.

If $b=0$, and $a, c\neq 0$; or if   $a=0$, and $b, c\neq 0$, then, after appropriate conjugation in $\mathrm{SL}(3, \mathbb{C})$, $L \sim \langle x_1, x_3, 2h_1+h_2+x_2\rangle=L^1_{3, -\frac{1}{4}}$.

If $c=0$, and $a, b\neq 0$, then, after appropriate conjugation in $\mathrm{SL}(3, \mathbb{C})$, $L$ is equivalent to  $\langle x_1, y_2, y_3\rangle$, which is 
equivalent to $L_{5}^1$.

\vspace{2mm}

\noindent \underline{Case $4$}. $L=\langle z \rangle \inplus K^4_1$ $=$ $\langle z \rangle \inplus \langle x_1, y_2 \rangle$.  The conditions $[z, x_1] \subseteq   \langle x_1, y_2 \rangle$, and  $[z,  y_2] \subseteq   \langle x_1, y_2\rangle$ imply that $L =\langle x_1, y_2,   ax_3+ by_3+ ch_1+dh_2\rangle$, $a, b, c, d \in \mathbb{C}$, not all zero.  
 In matrix form, we have
\begin{equation}
ax_3+ by_3+ ch_1+dh_2=  \left(
\begin{array}{ccc}
 c& 0 & -a   \\
 0 & -c+d &  0  \\
 -b & 0 & -d  
\end{array}
\right).
\end{equation}
We now consider the cases $4ab+c^2+2cd+d^2\neq 0$, and 
$4ab+c^2+2cd+d^2= 0$.

\vspace{2mm}

\noindent \underline{Case $4.1$}.  $4ab+c^2+2cd+d^2\neq 0$.  In this case $ax_3+ by_3+ ch_1+dh_2$ is diagonalizable, as is the $2\times 2$ matrix comprised of its four corners.
For any  matrix 
\begin{equation}\label{abovet}
A= \left(
\begin{array}{ccc}
g_{11} & 0& g_{12}    \\
 0 &  \frac{1}{\det(G)} & 0\\
 g_{21} &0& g_{22}
\end{array}
\right) \in \mathrm{SL}(3, \mathbb{C}),
\end{equation}
where $G=[g_{ij}] \in \mathrm{GL}(2, \mathbb{C})$,
$A^{-1} \langle x_1, y_2\rangle A=  \langle x_1, y_2\rangle $.  We may choose a matrix $A$ of the form above  in Eq. \eqref{abovet} which diagonalizes the $2\times 2$ matrix comprised of the four corners of  $ax_3+ by_3+ ch_1+dh_2$ such that  $A^{-1} \langle  ax_3+ by_3+ ch_1+dh_2 \rangle A$$=$ $\langle  h_1+\alpha h_2\rangle$, for some $\alpha \in \mathbb{C}$.  We now proceed in cases based on the value of
$\alpha$.

\vspace{2mm}

\noindent \underline{Case $4.1.1$}. $\alpha=-1$.  In this case $L= \langle x_1, y_2, h_1-h_2\rangle$, which is isomorphic to $L_{2}$ via the isomorphism given by
\begin{equation}\label{goodness2}
 \arraycolsep=1.4pt\def\arraystretch{1.5}
\begin{array}{llllll}
z_1 &\mapsto & x_1+y_2,\\
z_2 &\mapsto & x_1-y_2,\\
z_3 &\mapsto & \frac{1}{3}(h_1-h_2).
\end{array}
\end{equation}
Define
\begin{equation}
\begin{array}{llllll}
L_{2}^2 &\equiv& \langle x_1, y_2, h_1-h_2 \rangle.
\end{array}
\end{equation}
Straightforward computation shows that $L_{2}^1$, and $L_{2}^2$ are inequivalent.

\vspace{2mm}

\noindent \underline{Case $4.1.2$}. $\alpha=1$.   In this case $L= \langle x_1, y_2, h_1+h_2\rangle$, which is isomorphic to $L_{4}$ via the isomorphism given by
\begin{equation}
 \arraycolsep=1.4pt\def\arraystretch{1.5}
\begin{array}{llllll}
z_1 &\mapsto & x_1+y_2,\\
z_2 &\mapsto & x_1-y_2,\\
z_3 &\mapsto & h_1+h_2.
\end{array}
\end{equation}
Define
\begin{equation}
\begin{array}{llllll}
L_{4}^2 &\equiv& \langle x_1, y_2, h_1+h_2 \rangle.
\end{array}
\end{equation}
Straightforward computation shows that $L_{4}^1$ and $L_{4}^2$ are inequivalent.

\vspace{2mm}

\noindent \underline{Case $4.1.3$}. $\alpha\neq \pm 1$.   In this case the subalgebra $L= \langle x_1, y_2, h_1+\alpha h_2\rangle$ is isomorphic to  
$L_{3, -\frac{(2\alpha-1)(\alpha-2)}{9 (\alpha-1)^2} }$ via the isomorphism given by
\begin{equation}\label{goodness3}
 \arraycolsep=1.4pt\def\arraystretch{1.5}
\begin{array}{llllll}
z_1 &\mapsto & x_1+y_2,\\
z_2 &\mapsto & \frac{1}{3(\alpha-1)}( (\alpha-2) x_1+(2\alpha-1)y_2),\\
z_3 &\mapsto &-\frac{1}{3(\alpha-1)}(h_1+\alpha h_2).
\end{array}
\end{equation}
For $\alpha \neq \pm 1$, define
\begin{equation}
L_{3, -\frac{(2\alpha-1)(\alpha-2)}{9(\alpha-1)^2}}^{2, \alpha} \equiv \langle x_1, y_2, h_1+\alpha h_2  \rangle.
\end{equation}

For $\alpha, \beta \neq \pm 1$, we establish that $L_{3, -\frac{(2\alpha-1)(\alpha-2)}{9(\alpha-1)^2}}^{2, \alpha}$ is equivalent to $L_{3, -\frac{(2\beta-1)(\beta-2)}{9(\beta-1)^2}}^{2, \beta}$  if and only if $\alpha=\beta$, or $\alpha \beta=1$, as we did in Case $3.1.3$ (see Remark \ref{goodness} below for details).

We now establish that $L_{3, -\frac{(2\alpha-1)(\alpha-2)}{9(\alpha-1)^2}}^{2, \alpha}$ is not equivalent to any subalgebra defined previously in this proof.
First, we must consider $\alpha\neq \pm 1$ for which
$-\frac{(2\alpha-1)(\alpha-2)}{9(\alpha-1)^2}=-\frac{1}{4}, -\frac{2}{9}$, or $0$.  That is, we must consider $\alpha=0, 2$; as in Case $3.3.1$ we may exclude the
$\alpha=\frac{1}{2}$ case.  

Second, we must consider $L_{3, -\frac{(2\alpha-1)(\alpha-2)}{9(\alpha-1)^2}}^{2, \alpha}$, and $L_{3, -\frac{(2\alpha-1)(\alpha-2)}{9(\alpha-1)^2}}^{1, \alpha}$ for $\alpha \neq \pm 1$.  But, straightforward computation shows these two subalgebras to be inequivalent.

If $\alpha=0$, then $L_{3, -\frac{(2\alpha-1)(\alpha-2)}{9(\alpha-1)^2}}^{2, \alpha}=L_{3, -\frac{2}{9}}^{2, 0}$.  Straightforward calculation shows that
$L_{3, -\frac{2}{9}}^{2, 0}$ is inequivalent to $L_{3, -\frac{2}{9}}^{1}$, and $L_{3, -\frac{2}{9}}^{1,0}$.

If $\alpha=2$, then $L_{3, -\frac{(2\alpha-1)(\alpha-2)}{9(\alpha-1)^2}}^{2, \alpha}=L_{3,0}^{2, 2}$.  Straightforward calculation shows that
$L_{3, 0}^{2, 2}$ is inequivalent to $L_{3, 0}^{1}$, and $L_{3,0}^{1,2}$.

\vspace{2mm}

\noindent \underline{Case $4.2$}.  $4ab+c^2+2cd+d^2 = 0$.  Then $d=-c\pm 2\sqrt{-ab} $, so that
\begin{equation}
L=\langle x_1, y_2, ax_3+by_3+ch_1+ ( -c\pm 2\sqrt{-ab}  ) h_2 \rangle.
\end{equation}

\vspace{2mm}

\noindent \underline{Case $4.2.1$}. $b\neq  0$.  
Then conjugation by
\begin{equation}
A=\frac{1}{\sqrt[3]{-b}} \left(
\begin{array}{ccc}
0 & \pm\sqrt{-ab} &1   \\
 1 & 0 & 0   \\
 0& -b &0  
\end{array}
\right) \in \mathrm{SL}(3,\mathbb{C}),
\end{equation}
yields $\langle y_1, y_3, (-2c\pm 2\sqrt{-ab})h_1+(c\mp \sqrt{-ab})h_2+x_2\rangle$.  After an additional appropriate 
conjugation we get $L \sim \langle y_1, y_3, 2h_1+h_2+x_2\rangle$, or $\langle x_2, y_1, y_3\rangle$.  

The subalgebra $\langle y_1, y_3, 2h_1+h_2+x_2\rangle$ is isomorphic to $L_{3, -\frac{1}{4}}$ via the isomorphism given by
\begin{equation}
 \arraycolsep=1.4pt\def\arraystretch{1.5}
\begin{array}{llllll}
z_1 &\mapsto & y_1+y_3,\\
z_2 &\mapsto & \frac{2}{3}y_1+\frac{1}{2}y_3,\\
z_3 &\mapsto & -\frac{1}{6}(2h_1+h_2+x_2).
\end{array}
\end{equation}
Define
\begin{equation}
L^2_{3, -\frac{1}{4}} \equiv \langle y_1, y_3, 2h_1+ h_2+x_2  \rangle.
\end{equation}
Direct computation shows that $L^1_{3, -\frac{1}{4}}$, and $L^2_{3, -\frac{1}{4}}$  are inequivalent.

The subalgebra $\langle x_2, y_1, y_3\rangle$ is isomorphic to $L_{5}$ via the isomorphism given by
\begin{equation}
 \arraycolsep=1.4pt\def\arraystretch{1.5}
\begin{array}{llllll}
z_1 &\mapsto & x_2,\\
z_2 &\mapsto & y_1,\\
z_3 &\mapsto & y_3.
\end{array}
\end{equation}
The subalgebra $\langle x_2, y_1, y_3\rangle$ is equivalent to $L^1_5$ via conjugation by
\begin{equation}
 \left(
\begin{array}{ccc}
0 & 0 &1   \\
 1 & -1 & 0   \\
 0& 1 &-1  
\end{array}
\right) \in \mathrm{SL}(3,\mathbb{C}).
\end{equation}

\vspace{2mm}

\noindent \underline{Case $4.2.2$}. $b=  0$.  If $a=0$, then $L=\langle x_1, y_2, h_1-h_2\rangle=L^2_2$.
If $c=0$, then $L=\langle x_1, x_3, y_2\rangle$, which is
isomorphic to $L_{5}$ via the isomorphism given by
\begin{equation}
 \arraycolsep=1.4pt\def\arraystretch{1.5}
\begin{array}{llllll}
z_1 &\mapsto & y_2,\\
z_2 &\mapsto & -x_1,\\
z_3 &\mapsto & x_3.
\end{array}
\end{equation}
The subalgebra $\langle x_1, x_3, y_2\rangle$ is equivalent to $L^1_5$.

If $a, c \neq 0$, then, after appropriate conjugation, $L \sim \langle x_1, y_2, h_1-h_2+x_3\rangle$, which is
isomorphic to $L_{3, -\frac{1}{4}}$ via the isomorphism given by
\begin{equation}
 \arraycolsep=1.4pt\def\arraystretch{1.5}
\begin{array}{llllll}
z_1 &\mapsto & x_1+y_2,\\
z_2 &\mapsto & \frac{1}{3}x_1+\frac{1}{2}y_2,\\
z_3 &\mapsto & \frac{1}{6}(h_1-h_2+x_3).
\end{array}
\end{equation}
The subalgebra $\langle x_1, y_2, h_1-h_2+x_3\rangle$ is equivalent to $L^2_{3, -\frac{1}{4}}$.

\vspace{2mm}

\noindent \underline{Case $5$}. $L=\langle z \rangle \inplus K^5_1$ $=$ $\langle z \rangle \inplus \langle h_1, h_2 \rangle$.  The conditions $[z, h_1] \subseteq   \langle h_1, h_2 \rangle$, and  $[z,  h_2] \subseteq   \langle h_1, h_2\rangle$ imply that $z$ is linear combination of $h_1$ and $h_2$, which is a contradiction to the fact that $L$ is $3$-dimensional.  Hence, this case does not yield a subalgebra.

\vspace{2mm}

\noindent \underline{Case $6$}. $L=\langle z \rangle \inplus K^1_2$ $=$ $\langle z \rangle \inplus \langle x_1+x_2, h_1+h_2 \rangle$.  The conditions $[z, x_1+x_2] \subseteq   \langle x_1+x_2, h_1+h_2 \rangle$, and  $[z,  h_1+h_2] \subseteq   \langle x_1+x_2, h_1+h_2\rangle$ imply that $z$ is a linear combination of $x_1+x_2$, and $h_1+h_2$, which is a contradiction to the fact that $L$ is $3$-dimensional.  Hence, this case does not yield a subalgebra.

\vspace{2mm}

\noindent \underline{Case $7$}. $L=\langle z \rangle \inplus K^2_2$ $=$ $\langle z \rangle \inplus \langle x_1, -\frac{1}{3}h_1+\frac{1}{3}h_2+x_3 \rangle$. The conditions
$[z, x_1+x_2] \subseteq   \langle x_1+x_2, h_1+h_2 \rangle$, and  $[z,  h_1+h_2] \subseteq   \langle x_1+x_2, h_1+h_2\rangle$ imply that
$z$ is a linear combination of  $-\frac{1}{3}h_1+\frac{1}{3}h_2$, and $x_3$.  Hence,
$L=\langle x_1, x_3, h_1-h_2\rangle$, which is isomorphic to $L_{3,0}$  via the isomorphism given by
\begin{equation}
 \arraycolsep=1.4pt\def\arraystretch{1.3}
\begin{array}{llllll}
z_1 &\mapsto & x_1+x_3,\\
z_2 &\mapsto & x_1,\\
z_3 &\mapsto & \frac{1}{3}(h_1-h_2).
\end{array}
\end{equation}
$L^{1,2}_{3,0}$ is equivalent to $\langle x_1, x_3, h_1- h_2  \rangle$ via conjugation by
\begin{equation}
 \left(
\begin{array}{ccc}
1 & 1 &-\frac{1}{3}   \\
 0 & 1 & 0   \\
 0& 0 &1  
\end{array}
\right) \in \mathrm{SL}(3,\mathbb{C}).
\end{equation}

\vspace{2mm}

\noindent \underline{Case $8$}. $L=\langle z \rangle \inplus K^3_2$ $=$ $\langle z \rangle \inplus \langle x_1, -\frac{2}{3}h_1-\frac{1}{3}h_2+y_2 \rangle$.  The conditions $[z, x_1] \subseteq   \langle x_1, -\frac{2}{3}h_1-\frac{1}{3}h_2+y_2 \rangle$, and  $[z,  -\frac{2}{3}h_1-\frac{1}{3}h_2+y_2] \subseteq   \langle x_1, -\frac{2}{3}h_1-\frac{1}{3}h_2+y_2 \rangle$ imply that
$z=a( 2h_1+h_2)+bx_1+cy_2$, for $a, b, c \in \mathbb{C}$.  
The requirement that $\dim L = 3$ then 
 implies $L=\langle x_1, y_2, 2h_1+h_2\rangle$, which is isomorphic to $L_{3,0}$  via the isomorphism given by
\begin{equation}
 \arraycolsep=1.4pt\def\arraystretch{1.5}
\begin{array}{llllll}
z_1 &\mapsto & x_1+y_2,\\
z_2 &\mapsto & x_1,\\
z_3 &\mapsto & \frac{1}{3}(2h_1+h_2).
\end{array}
\end{equation}
We have $\langle x_1, y_2, 2h_1+h_2  \rangle=L^{2,\frac{1}{2}}_{3,0}$.

\vspace{2mm}

\noindent \underline{Case $9$}. $L=\langle z \rangle \inplus K^{4, \alpha}_2$ $=$ $\langle z \rangle \inplus \langle x_1, \alpha h_1+(2\alpha+1)h_2 \rangle$.  The conditions $[z, x_1] \subseteq   \langle x_1, \alpha h_1+(2\alpha+1)h_2 \rangle$, and  $[z, \alpha h_1+(2\alpha+1)h_2] \subseteq   \langle x_1, \alpha h_1+(2\alpha+1)h_2 \rangle$ imply that
$L=\langle x_1, \alpha h_1+(2\alpha+1)h_2, ax_3+by_2+ch_1+dh_2\rangle$, $a, b, c, d \in \mathbb{C}$, such that  $a(3\alpha+1)=0$, and $b(3\alpha+2)=0$.
  
 The value of $\alpha$ ($\alpha=-\frac{1}{3}$; $\alpha=-\frac{2}{3}$; or $\alpha \neq -\frac{1}{3}, -\frac{2}{3}$) yields the following 
equivalence possibilities for $L$:
\begin{equation}
 \arraycolsep=1.4pt\def\arraystretch{1.3}
\begin{array}{lllllllllll}
 \langle x_1, x_3, h_1-h_2\rangle&\sim & L^{1,2}_{3,0}, \\
  \langle x_1, y_2, 2h_1+h_2\rangle&=& L^{2,\frac{1}{2}}_{3,0},  \\
   \langle x_1, h_1, h_2\rangle&=&L^1_{3,0}.
\end{array}
\end{equation}
\end{proof}

\begin{remark}\label{goodness}
A recent article \cite{chap}, which appeared just before the present version of this article, states that there are two inaccuracies in the above classification (see also the previous version of this article \cite{a2}). However, this appears to be incorrect. First, 
in \cite{chap}, it's claimed that the subalgebra  $\big\langle h_1+2h_2, x_2, x_3 \big\rangle$, which the authors denote $f_{3.2}$,  is missing from the above classification of Theorem \ref{threeds}, see also \cite{a2} (Note that we have expressed this subalgebra in the basis of $\sll$ used in this article). But, this isn't correct; the subalgebra does appear in our classification. To be specific, $f_{3.2} \sim L^{2}_{2}=\langle x_1, y_2, h_1-h_2 \rangle$, since
$G  \langle h_1+2h_2, x_2, x_3 \rangle G^{-1} = L^2_2$, where
\begin{equation}
G= \left(
\begin{array}{ccc}
-1 & 1 &0  \\
 0 & 0 & 1   \\
 0& 1 &0  
\end{array}
\right) \in \mathrm{SL}(3,\mathbb{C}).
\end{equation}

Before moving on to our second point, we mention   an additional related inaccuracy  in \cite{chap}. In \cite{chap}, the authors state that $f_{3.2}$ would be equivalent to $L^{2,-1}_{3,-\frac{1}{4}}$. However, this misidentifies the isomorphism type--which is determined by the subscript of $L$--of the subalgebra, which should
be of isomorphism type $L_2$, and not of isomorphism type $L_{3,-\frac{1}{4}}$  (see Eqs. (\ref{Eq:deg3}, \ref{goodness2}, \ref{goodness3})).

Now we move on to our second point. The  authors of \cite{chap} claim that the constraints of the following subalgebra of Theorem \ref{threeds} (see also \cite{a2}) aren't correct:
\begin{equation}
 \arraycolsep=1.4pt\def\arraystretch{1.5}
\begin{array}{llllll}
 L^{2, \alpha}_{3,\Psi(\alpha)}= \langle x_1, y_2, h_1+ \alpha h_2 \rangle, \alpha \neq \pm 1, ~\text{such that}\\
 L^{2, \alpha}_{3,\Psi(\alpha)} \sim  L^{2, \beta}_{3,\Psi(\beta)} ~\text{ if and only if}~ \alpha=\beta ~\text{or}~ \alpha \beta=1,\\
\text{ where} ~\Psi(\alpha)=-\frac{(2\alpha-1)(\alpha-2)}{9(\alpha-1)^2}.
 \end{array}
\end{equation}
However, we now illustrate that our constraints and conditions for the parameter of $ L^{2, \alpha}_{3,\Psi(\alpha)}$ are in fact correct.

To begin, note that the subalgebras corresponding to $\alpha = \pm 1$ are included in the classification (i.e., $L^2_2 =\langle x_1, y_2, h_1-h_2\rangle$ and $L^2_4=\langle x_1, y_2, h_1+h_2\rangle$).
Now, let  $\alpha \beta=1$. Then $ L^{2, \alpha}_{3,\Psi(\alpha)} \sim  L^{2, \beta}_{3,\Psi(\beta)}$, since $G(L^{2, \alpha}_{3,\Psi(\alpha)} ) G^{-1}=L^{2, \beta}_{3,\Psi(\beta)}$, where
\begin{equation}
G= \left(
\begin{array}{ccc}
0 & 0 &1  \\
 0 & 1 & 0   \\
 -1& 0 &0  
\end{array}
\right) \in \mathrm{SL}(3,\mathbb{C}).
\end{equation}

Now suppose that $L_{3, -\frac{(2\alpha-1)(\alpha-2)}{9(\alpha-1)^2}}^{2, \alpha}$ is equivalent to $L_{3, -\frac{(2\beta-1)(\beta-2)}{9(\beta-1)^2}}^{2, \beta}$, then these subalgebras are  isomorphic so that   $-\frac{(2\alpha-1)(\alpha-2)}{9(\alpha-1)^2}=-\frac{(2\beta-1)(\beta-2)}{9(\beta-1)^2}$. It follows that
\begin{equation}\label{wowgb}
 \arraycolsep=1.4pt\def\arraystretch{1.3}
\begin{array}{lllll}
(2\alpha-1)(\alpha-2)(\beta-1)^2=(2\beta-1)(\beta-2)(\alpha-1)^2.
\end{array}
\end{equation}
Expanding, and simplifying Eq. \eqref{wowgb}  yields
\begin{equation}
 \arraycolsep=1.4pt\def\arraystretch{1.5}
\begin{array}{lllll}
&\alpha^2\beta-\alpha \beta^2-\alpha+\beta=0 \\
\Rightarrow & (\alpha \beta -1)(\alpha-\beta)=0\\
\Rightarrow & \alpha \beta =1 ~\text{or}~\alpha=\beta, ~\text{as required}.
\end{array}
\end{equation}
\end{remark}

\begin{remark}
The (complexification) of the Euclidean algebra $\mathfrak{e}(2)\cong \mathfrak{so}(2) \inplus \mathbb{R}^2$ is isomorphic to $L_4$.  Hence, by Theorem \ref{threeds},  
there are precisely two nonequivalent copies of 
(the complexification of) $\mathfrak{e}(2)$
in $\sll$.  This point is also implied by the results of  the authors in \cite{drj}.  
\end{remark}


\subsection{Four-dimensional, solvable subalgebras of $\sll$}

\begin{theorem}
A classification of the  $4$-dimensional solvable subalgebras of $\sll$, up to inner automorphism, is given by:
\begin{equation}
 \arraycolsep=1.4pt\def\arraystretch{1.2}
\begin{array}{lllllll}
M^{1}_{8} &=& \langle x_1, x_3, h_1, h_2 \rangle,\\
M^{2}_{8} &=& \langle x_1, y_2, h_1, h_2 \rangle,\\
M^{1}_{12} &=& \langle x_1, x_2, x_3,  h_1+h_2 \rangle,\\
M^{1, \alpha}_{13, \frac{(2\alpha-1)(\alpha-2)}{(\alpha+1)^2}} &=& \langle x_1, x_2, x_3, \alpha h_1+ h_2 \rangle,\\
 M^{2}_{13, 2} &=&  \langle x_1, x_2, x_3,  h_1 \rangle,\\
M^{1}_{14} &=& \langle x_1, x_2, x_3,  h_1-h_2 \rangle,\\
\end{array}
\end{equation}
where $\alpha \neq \pm 1$; and  $M^{1, \alpha}_{13, \frac{(2\alpha-1)(\alpha-2)}{(\alpha+1)^2}}$  and 
$M^{1, \beta}_{13, \frac{(2\beta-1)(\beta-2)}{(\beta+1)^2}}$  are equivalent if and only if $\alpha=\beta$. The subscripts correspond to those of the appropriate 
isomorphism type in Equation (\ref{Eq:deg32}).
\end{theorem}
\begin{proof}
Let $M$ be a solvable $4$-dimensional subalgebra of $\sll$.  Then, by Lemma \ref{graaff}, 
\begin{equation}
M=\langle z \rangle \inplus L,
\end{equation}
where $L$ is a $3$-dimensional solvable subalgebra, and $z\in M$.  By Theorem \ref{threeds}, we have the following cases.

\vspace{2mm}

\noindent \underline{Case 1}.  $M=\langle z\rangle \inplus  L^1_{3,-\frac{2}{9}}$ $=$ $\langle z\rangle \inplus \langle x_1+x_2, x_3, h_1+h_2 \rangle$.  
Then, $z$ is a linear combination of $x_1+x_2$, $x_3$, and $h_1+h_2$, which is a contradiction to the dimension of $M$.  Hence, this case does not yield a subalgebra.

\vspace{2mm}

\noindent \underline{Case 2}.  $M=\langle z\rangle \inplus  L^1_{3,0}$ $=$ $\langle z\rangle \inplus \langle x_1, h_1, h_2 \rangle$.  
Then, $z$ is a linear combination of $x_1$, $h_1$, and $h_2$, which is a contradiction to the dimension of $M$.  Hence, this case does not yield a subalgebra.

\vspace{2mm}

\noindent \underline{Case 3}.  $M=\langle z\rangle \inplus  L^{1, \alpha}_{3,-\frac{(2\alpha-1)(\alpha-2)}{9(\alpha-1)^2}}$ $=$ $\langle z\rangle \inplus \langle x_1, x_3, (\alpha-1)h_1+\alpha h_2 \rangle$, $\alpha \neq \pm 1$.  
Then, $M=\langle x_1, x_3, h_1, h_2\rangle$, which is isomorphic to $M_{8}$ via the isomorphism given by
\begin{equation}
 \arraycolsep=1.4pt\def\arraystretch{1.3}
\begin{array}{llllll}
z_1 &\mapsto & \frac{1}{3}(h_1-h_2),\\
z_2 &\mapsto & x_1,\\
z_3 &\mapsto & \frac{1}{3}(h_1+2h_2),\\
z_4 & \mapsto & x_3.
\end{array}
\end{equation}
Define
\begin{equation}
M^{1}_{8} \equiv \langle x_1, x_3, h_1, h_2 \rangle.
\end{equation}

\vspace{2mm}

\noindent \underline{Case 4}.  $M=\langle z\rangle \inplus  L^{2, \alpha}_{3,-\frac{(2\alpha-1)(\alpha-2)}{9(\alpha-1)^2}}$ $=$ $\langle z\rangle \inplus \langle x_1, y_2, h_1+\alpha h_2 \rangle$, $\alpha \neq \pm 1$.  
Then, $M=\langle x_1, y_2, h_1, h_2\rangle$, which is isomorphic to $M_{8}$ via the isomorphism given by
\begin{equation}
 \arraycolsep=1.4pt\def\arraystretch{1.3}
\begin{array}{llllll}
z_1 &\mapsto & -\frac{1}{3}(h_1+2h_2),\\
z_2 &\mapsto & y_2,\\
z_3 &\mapsto & \frac{1}{3}(2h_1+h_2),\\
z_4 & \mapsto & x_1.
\end{array}
\end{equation}
Define
\begin{equation}
M^{2}_{8} \equiv \langle x_1, y_2, h_1, h_2 \rangle.
\end{equation}
Straightforward computation establishes that $M^{1}_{8}$ and $M^{2}_{8}$ are inequivalent.

\vspace{2mm}

\noindent \underline{Case 5}.  $M=\langle z\rangle \inplus  L^1_{4}$ $=$ $\langle z\rangle \inplus \langle x_1, x_3, h_2 \rangle$.  
Then, $M=\langle x_1, x_3, h_1, h_2\rangle=M^{1}_8$.

\vspace{2mm}

\noindent \underline{Case 6}.  $M=\langle z\rangle \inplus  L^2_{4}$ $=$ $\langle z\rangle \inplus \langle x_1, y_2, h_1+h_2 \rangle$.  
Then, $M=\langle x_1, y_2, h_1, h_2\rangle=M^{2}_8$.

\vspace{2mm}

\noindent \underline{Case 7}.  $M=\langle z\rangle \inplus  L^2_{4}$ $=$ $\langle z\rangle \inplus \langle x_1, x_2, x_3 \rangle$.  
Then,  $M=\langle x_1, x_2, x_3, \alpha h_1+ h_2\rangle$ or $\langle x_1, x_2, x_3,  h_1+ \alpha h_2\rangle$, for some $\alpha \in \mathbb{C}$. 
Up to equivalence, this reduces to $\langle x_1, x_2, x_3, \alpha h_1+ h_2\rangle$ for $\alpha \in \mathbb{C}$ or $\langle x_1, x_2, x_3,  h_1\rangle$.

For $\alpha \neq \pm 1$, $\langle x_1, x_2, x_3, \alpha h_1+ h_2\rangle$ is isomorphic to $M_{13, \frac{(2\alpha-1)(\alpha-2)}{(\alpha+1)^2}}$ via the isomorphism given by
\begin{equation}
 \arraycolsep=1.4pt\def\arraystretch{1.3}
\begin{array}{llllll}
z_1 &\mapsto & \frac{2\alpha-1}{\alpha+1}x_1-\frac{\alpha-2}{\alpha+1}x_2,\\
z_2 &\mapsto & \frac{3(\alpha-1)}{\alpha+1}x_3,\\
z_3 &\mapsto & x_1+x_2,\\
z_4 & \mapsto & \frac{1}{\alpha+1}(\alpha h_1+h_2).
\end{array}
\end{equation}
For $\alpha \neq \pm 1$, define
\begin{equation}
M^{1, \alpha}_{13, \frac{(2\alpha-1)(\alpha-2)}{(\alpha+1)^2}} \equiv \langle x_1, x_2, x_3, \alpha h_1+ h_2 \rangle.
\end{equation}
Direct, straightforward computation establishes that $M^{1, \alpha}_{13, \frac{(2\alpha-1)(\alpha-2)}{(\alpha+1)^2}}$  and 
$M^{1, \beta}_{13, \frac{(2\beta-1)(\beta-2)}{(\beta+1)^2}}$  are equivalent if and only if $\alpha=\beta$.

The subalgebra $\langle x_1, x_2, x_3, h_1\rangle$ is isomorphic to $M_{13, 2}$ via the isomorphism given by
\begin{equation}
 \arraycolsep=1.4pt\def\arraystretch{1.3}
\begin{array}{llllll}
z_1 &\mapsto & 2x_1-x_2,\\
z_2 &\mapsto & 3x_3,\\
z_3 &\mapsto &x_1+x_2,\\
z_4 & \mapsto & h_1.
\end{array}
\end{equation}
Define
\begin{equation}
M^{2}_{13, 2} \equiv \langle x_1, x_2, x_3,  h_1 \rangle.
\end{equation}

Then
$M^{2}_{13, 2}$ and $M^{1, \alpha}_{13, \frac{(2\alpha-1)(\alpha-2)}{(\alpha+1)^2}}$ are isomorphic precisely when $\alpha=0$.  Straightforward computation
shows that $M^{2}_{13, 2}$ and $M^{1, 0}_{13, 2}$ are inequivalent.

The subalgebra $\langle x_1, x_2, x_3, h_1+h_2\rangle$ is isomorphic to $M_{12}$ via the isomorphism given by
\begin{equation}
 \arraycolsep=1.4pt\def\arraystretch{1.3}
\begin{array}{llllll}
z_1 &\mapsto & 2x_1+x_2,\\
z_2 &\mapsto & x_3,\\
z_3 &\mapsto & x_1+x_2,\\
z_4 & \mapsto & h_1+h_2.
\end{array}
\end{equation}
Define
\begin{equation}
M^{1}_{12} \equiv \langle x_1, x_2, x_3,  h_1+h_2 \rangle.
\end{equation}

The subalgebra $\langle x_1, x_2, x_3, h_1-h_2\rangle$ is isomorphic to $M_{14}$ via the isomorphism given by
\begin{equation}
 \arraycolsep=1.4pt\def\arraystretch{1.3}
\begin{array}{llllll}
z_1 &\mapsto & x_1-x_2,\\
z_2 &\mapsto & 2x_3,\\
z_3 &\mapsto & x_1+x_2,\\
z_4 & \mapsto & \frac{1}{3}(h_1-h_2).
\end{array}
\end{equation}
Define
\begin{equation}
M^{1}_{14} \equiv \langle x_1, x_2, x_3,  h_1-h_2 \rangle.
\end{equation}

\vspace{2mm}

\noindent \underline{Case 8}.  $M=\langle z\rangle \inplus  L^1_{3,-\frac{1}{4}}$ $=$ $\langle z\rangle \inplus \langle x_1, x_3, 2h_1+h_2+x_2 \rangle$.  
Then, $M=\langle x_1, x_2, x_3, 2h_1+h_2\rangle=M^{1,2}_{13,0}$.

\vspace{2mm}

\noindent \underline{Case 9}.  $M=\langle z\rangle \inplus  L^2_{3,-\frac{1}{4}}$ $=$ $\langle z\rangle \inplus \langle y_1, y_3, 2h_1+h_2+x_2 \rangle$.  
Then, $M=\langle x_2, y_1, y_3, 2h_1+h_2\rangle$, which is isomorphic to $M_{13, 0}$ via the isomorphism given by
\begin{equation}
 \arraycolsep=1.4pt\def\arraystretch{1.3}
\begin{array}{llllll}
z_1 &\mapsto & y_3,\\
z_2 &\mapsto & -y_1,\\
z_3 &\mapsto & x_2+y_3,\\
z_4 & \mapsto & -\frac{1}{3}(2h_1+h_2).
\end{array}
\end{equation}
Straightforward computation establishes that $\langle x_2, y_1, y_3, 2h_1+h_2  \rangle$ is equivalent to $M^{1,\frac{1}{2}}_{13, 0}$.
See comments after this proof for details.

\vspace{2mm}

\noindent \underline{Case 10}.  $M=\langle z\rangle \inplus  L^1_2$ $=$ $\langle z\rangle \inplus \langle x_1, x_3, 2h_1+h_2 \rangle$.  
Then, $M=\langle x_1, x_3, 2h_1+h_2, ax_2+by_2+ch_1+dh_2\rangle$, for $a, b, c, d \in \mathbb{C}$, not all zero.  In matrix form $ax_2+by_2+ch_1+dh_2$ is equal to  the block matrix
\begin{equation}
ax_2+by_2+ch_1+dh_2= \left(
\begin{array}{ccc}
c & 0 & 0   \\
 0 & -c+d & a   \\
 0& b &-d  
\end{array}
\right).
\end{equation}

\vspace{2mm}

\noindent \underline{Case 10.1}.  $4ab+c^2-4cd+4d^2\neq 0$.   The eigenvalues of the lower right $2\times 2$ block of $ax_2+by_2+ch_1+dh_2$ are 
\begin{equation}
 \arraycolsep=1.4pt\def\arraystretch{1.3}
\begin{array}{llll}
-\frac{1}{2}c+\frac{1}{2}\sqrt{4ab+c^2-4cd+4d^2}, ~\text{and}\\
-\frac{1}{2}c-\frac{1}{2}\sqrt{4ab+c^2-4cd+4d^2},
\end{array}
\end{equation}
which are distinct for $4ab+c^2-4cd+4d^2\neq 0$.  Hence,  $ax_2+by_2+ch_1+dh_2$ is diagonalizable.   Note that the lower right $2\times 2$ blocks of $2h_1+h_2$ and $ax_2+by_2+ch_1+dh_2$ commute.  Hence, the lower right $2\times 2$ blocks of these elements are simultaneously diagonalizable, via conjugation by some     $G\in \mathrm{GL}(2, \mathbb{C})$.  If we set
\begin{equation}
A= \left(
\begin{array}{ccc}
\frac{1}{\det(G)} & 0    \\
 0 &  G
\end{array}
\right) \in \mathrm{SL}(3, \mathbb{C}),
\end{equation}
then $A^{-1}\langle x_1, x_3\rangle A=\langle x_1, x_3 \rangle$, so that $A^{-1}\langle x_1, x_3, 2h_1+h_2, ax_2+by_2+ch_1+dh_2\rangle A$$=$$\langle x_1, x_3, h_1, h_2 \rangle$
$=$ $M^1_8$.

\vspace{2mm}

\noindent \underline{Case 10.2}.  $4ab+c^2-4cd+4d^2= 0$.  Then $d= \frac{1}{2}c\pm \sqrt{-ab}$.

\vspace{2mm}

\noindent \underline{Case 10.2.1}. $b\neq 0$.  Let
\begin{equation}
G= \left(
\begin{array}{ccc}
\pm \sqrt{-ab} & 1    \\
 b &  0
\end{array}
\right)\in \mathrm{GL}(2, \mathbb{C}),
\end{equation}
where the sign in $G$ varies with the sign in $d$.  Then conjugation by
\begin{equation}
A= \left(
\begin{array}{ccc}
\frac{1}{\det(G)} & 0    \\
 0 &  G
\end{array}
\right) \in \mathrm{SL}(3, \mathbb{C}),
\end{equation}
yields $\langle x_1, x_3, 2h_1+h_2, ch_1+\frac{1}{2}ch_2+x_2\rangle$ $=$ $\langle x_1, x_2, x_3, 2h_1+h_2 \rangle$$=$ $M^{1,2}_{13, 0}$.

\vspace{2mm}

\noindent \underline{Case 10.2.2}. $b= 0$.  Then $M =\langle x_1, x_3, 2h_1+h_2, ch_1+\frac{1}{2}ch_2+a x_2\rangle$, and, by dimension considerations, $a\neq 0$.  Hence, 
$M=\langle x_1, x_2, x_3, 2h_1+h_2 \rangle$ $=$ $M^{1,2}_{13, 0}$.

\vspace{2mm}

\noindent \underline{Case 11}.  $M=\langle z\rangle \inplus  L^2_2$ $=$ $\langle z\rangle \inplus \langle x_1, y_2, h_1-h_2 \rangle$.  
Then, $M=\langle x_1, y_2, h_1-h_2, ax_3+by_3+ch_1+dh_2\rangle$, for $a, b, c, d \in \mathbb{C}$, not all zero.  In matrix form,        
$ax_3+by_3+ch_1+dh_2$ is equal to 
\begin{equation}
ax_3+by_3+ch_1+dh_2= \left(
\begin{array}{ccc}
c & 0 & -a   \\
 0 & -c+d & 0   \\
 -b& 0 &-d  
\end{array}
\right).
\end{equation}

\vspace{2mm}

\noindent \underline{Case 11.1}.  $4ab+c^2+2cd+d^2\neq 0$.  The eigenvalues of the $2\times 2$ matrix formed from the outer corners of $ax_3+by_3+ch_1+dh_2$ are 
\begin{equation}
 \arraycolsep=1.4pt\def\arraystretch{1.3}
\begin{array}{llll}
\frac{1}{2}c-\frac{1}{2}d+\frac{1}{2}\sqrt{4ab+c^2+2cd+d^2}, ~\text{and}\\
\frac{1}{2}c-\frac{1}{2}d-\frac{1}{2}\sqrt{4ab+c^2+2cd+d^2},
\end{array}
\end{equation}
which are distinct for $4ab+c^2+2cd+d^2\neq 0$.  Hence,  $ax_3+by_3+ch_1+dh_2$ is diagonalizable.
Note that the  $2\times 2$ matrices formed from the corners  of $h_1-h_2$ and $ax_3+by_3+ch_1+dh_2$ commute.  Hence, the  $2\times 2$ matrices formed from the corners of these elements are simultaneously diagonalizable, via conjugation by some   $G=[g_{ij}] \in \mathrm{GL}(2, \mathbb{C})$.  If we set
\begin{equation}
A= \left(
\begin{array}{ccc}
g_{11}& 0 & g_{12} \\
0& \frac{1}{\det(G)} & 0    \\
 g_{21} & 0 & g_{22}  
\end{array}
\right) \in \mathrm{SL}(3, \mathbb{C}),
\end{equation}
then $A^{-1}\langle x_1, y_2\rangle A=\langle x_1, y_2 \rangle$, so that $A^{-1}\langle x_1, y_2, h_1-h_2, ax_3+by_3+ch_1+dh_2\rangle A$$=$$\langle x_1, y_2, h_1, h_2 \rangle$
$=$ $M^{2}_{8}$.

\vspace{2mm}

\noindent \underline{Case 11.2}.  $4ab+c^2+2cd+d^2= 0$.  Then $d=-c\pm2\sqrt{-ab}$.

\vspace{2mm}

\noindent \underline{Case 11.2.1}. $b\neq 0$.  Let
\begin{equation}
G=[g_{ij}]= \left(
\begin{array}{ccc}
\pm \sqrt{-ab} & 1    \\
-b &  0
\end{array}
\right)\in \mathrm{GL}(2, \mathbb{C}),
\end{equation}
where the sign in $G$ varies with the sign in $d$.  Then conjugation by
\begin{equation}
A= \left(
\begin{array}{ccc}
g_{11}& 0 & g_{12} \\
0& \frac{1}{\det(G)} & 0    \\
 g_{21} & 0 & g_{22}  
\end{array}
\right) \in \mathrm{SL}(3, \mathbb{C}),
\end{equation}
yields $\langle x_1, y_2, h_1-h_2, (c \mp \sqrt{-ab}) h_1-(c\mp   \sqrt{-ab})h_2+x_3\rangle$ $=$ $\langle x_1, x_3, y_2, h_1-h_2 \rangle$,  which is equivalent to $M^{1,\frac{1}{2}}_{13,0}$. 

\vspace{2mm}

\noindent \underline{Case 11.2.2}.  $b=0$. Then $M =\langle x_1, y_2, h_1-h_2, ch_1-ch_2+a x_3\rangle$, and, by dimension considerations, $a\neq 0$.  Hence, 
$M=\langle x_1, x_3, y_2, h_1-h_2 \rangle$,  which is equivalent to $M^{1,\frac{1}{2}}_{13,0}$. 
\end{proof}

Note that in the above theorem,  we addressed one oversight from the previous version of this article \cite{a2}. Specifically, the previous version contained the following two four-dimensional, solvable subalgebras: 
$\big\langle x_1, x_2, x_{3}, \frac{1}{2}h_1+h_2  \big\rangle$ and $\langle x_2, y_1, y_{3}, 2h_1+h_2 \rangle$. However, these subalgebras
are conjugate in $\mathrm{SL}_3(\mathbb{C})$.  In particular, $G\big\langle x_1, x_2, x_3,  \frac{1}{2}h_1+h_2 \big\rangle G^{-1}=\langle x_2, y_1, y_3, 2h_1+h_2 \rangle$, where
\begin{equation}
G=\left(
\begin{array}{rrr}
0&0&1 \\
-1&0&0 \\
0&-1&0 
\end{array}
\right)\in \mathrm{SL}(3,\mathbb{C}).
\end{equation} The equivalence of these subalgebras seems to also have been noted in \cite{chap}. In the theorem  above, we exclude the later subalgebra.

\subsection{Five-dimensional, solvable subalgebras of $\sll$}

Since simple Lie algebras have unique, up to inner automorphism, maximal solvable subalgebras, Borel subalgebras, we have the following theorem.
\begin{theorem}
The unique $5$-dimensional solvable subalgebra of $\sll$, up to inner automorphism, is a Borel subalgebra:
\begin{equation}
B
\equiv \langle x_1,x_2, x_3, h_1, h_2\rangle.
\end{equation}
\end{theorem}

\section{Conclusions}

A classification of the semisimple subalgebras of 
$A_{2}$, 
the Lie algebra of traceless $3\times 3$ matrices with complex entries, 
is well-known.  In this article, we classified  the solvable and Levi decomposable subalgebras of $\sll$, up to inner automorphism.  By Levi's Theorem, this completes the  classification of the subalgebras of $\sll$.  The classification is summarized in Table \ref{enddd}.

\begin{landscape}
\begin{small}
\begin{table} [!h]\renewcommand{\arraystretch}{1.5} \caption{Classification of subalgebras of $\sll$, up to inner automorphism.} \label{enddd}\begin{center}
\begin{tabular}{|c|c|c|c|c|c|c|} 
\hline
Dimension &   Semisimple  & Levi decomposable & Solvable\\
\hline \hline
1 & None & None &  $J^1$,  $J^{2}$, $J^3$, $J^{4,\alpha}$ \\
 &  &  &  where  $J^{4,\alpha} \sim  J^{4,\beta}$ if and only if $\alpha = 1-\beta$,  $\alpha =\beta$, \\
&&& $\alpha =(-1+\alpha)\beta$,  $\alpha =-(-1+\alpha)(-1+\beta)$,  \\
&&& $\alpha \beta=1$, or $-\alpha (-1+\beta)=1$.
 \\
\hline
2 & None& None &$K^i_1$, $1\leq i \leq 5$, $K^1_2$, $K^2_2$, $K^3_2$, $K^{4,\alpha}_2$ \\
 &  &  &   where $K^{4,\alpha}_2 \sim K^{4,\beta}_2$ if and only if $\alpha=\beta$.\\
\hline
3 & $A_1^1, A_1^2$ & None &  $L^1_2$, $L^2_2$,  $L^1_{3, -\frac{1}{4}}$, $L^2_{3, -\frac{1}{4}}$, $L^1_{3, -\frac{2}{9}}$, $L^1_{3,0}$, $L^1_4$, $L^2_4$,  $L^1_5$, \\
&&& $L^{1,\alpha}_{3, -\frac{(2\alpha-1)(\alpha-2)}{9(\alpha-1)^2}}$,
$L^{2,\alpha}_{3, -\frac{(2\alpha-1)(\alpha-2)}{9(\alpha-1)^2}}$, $\alpha \neq \pm 1$,\\
&&& where, for $i=1, 2$,   $L^{i,\alpha}_{3, -\frac{(2\alpha-1)(\alpha-2)}{9(\alpha-1)^2}} \sim L^{i,\beta}_{3, -\frac{(2\beta-1)(\beta-2)}{9(\beta-1)^2}}$\\
&&& if and only if $\alpha=\beta$ or $\alpha\beta=1$.\\
\hline
4 & None & $(A_1 \oplus_{} J)^1 $&  $M^{1}_{8}$, $M^{2}_{8}$, $M^{1}_{12}$,  $M^{2}_{13, 2}$, $M^{1}_{14}$, $M^{1, \alpha}_{13, \frac{(2\alpha-1)(\alpha-2)}{(\alpha+1)^2}}$, $\alpha \neq \pm 1$,\\
&&& where $M^{1, \alpha}_{13, \frac{(2\alpha-1)(\alpha-2)}{(\alpha+1)^2}}$ $\sim$ $M^{1, \beta}_{13, \frac{(2\beta-1)(\beta-2)}{(\beta+1)^2}}$ if and only if $\alpha=\beta$.\\
\hline
5 & None & $(A_1 \inplus_{\varphi_1} K_1)^1 $,  $(A_1 \inplus_{\varphi_1} K_1)^2$&  
$B$
\\
\hline
6 & None & $(A_1 \inplus_{\varphi_2} L_{2})^1 $, $(A_1 \inplus_{\varphi_2} L_{2})^2$& None \\
\hline
\end{tabular}\end{center}
\end{table}
\end{small}
\end{landscape}

\section*{Acknowledgements}

The work of A.D. was partially supported by research grants from the Professional Staff
Congress/City University of New York (Grant Nos. TRADA-46-165, TRADA-45-174, and TRADA-44-101). The
work of J.R. was partially supported by the Natural Sciences and Engineering Research Council
(Grant No. 3166-09).

\end{document}